\documentclass{article}%
\pdfoutput=1
\usepackage{amsmath}
\usepackage{amsfonts}
\usepackage{amssymb}
\usepackage{graphicx}%
\newtheorem{theorem}{Theorem}
\newtheorem{acknowledgement}[theorem]{Acknowledgement}

\newtheorem{corollary}[theorem]{Corollary}

\newtheorem{definition}[theorem]{Definition}

\newtheorem{notation}[theorem]{Notation}

\newtheorem{proposition}[theorem]{Proposition}
\newtheorem{remark}[theorem]{Remark}

\newenvironment{proof}[1][Proof]{\noindent\textbf{#1.} }{\ \rule{0.5em}{0.5em}}

\begin{document}

\title{Discrete conformal variations and scalar curvature on piecewise flat two and
three dimensional manifolds}
\author{David Glickenstein\thanks{Partially supported by NSF grant DMS 0748283}\\University of Arizona}
\maketitle

\section{Introduction}

Consider a manifold constructed by identifying the boundaries of Euclidean
triangles or Euclidean tetrahedra. When these form a closed topological
manifold, we call such spaces piecewise flat manifolds (see Definition
\ref{definition piecewise flat}) as in \cite{CMS}. Such spaces may be
considered discrete analogues of Riemannian manifolds, in that their geometry
can be described locally by a finite number of parameters, and the study of
curvature on such spaces goes back at least to Regge \cite{Reg}. In this
paper, we give a definition of conformal variation of piecewise flat manifolds
in order to study the curvature of such spaces.

Conformal variations of Riemannian manifolds have been well studied. While the
most general variation formulas for curvature quantities is often complicated,
the same formulas under conformal variations often take a simpler form. For
this reason, it has even been suggested that an approach to finding Einstein
manifolds would be to first optimize within a conformal class, finding a
minimum of the Einstein-Hilbert functional within that conformal class, and
then maximize across conformal classes to find a critical point of the
functional in general \cite{And}. Finding critical points of the
Einstein-Hilbert functional within a conformal class is a well-studied problem
dating back to Yamabe \cite{Yam}, and the proof that there exists a constant
scalar curvature metric in every conformal class was completed by Trudinger
\cite{Tru}, Aubin \cite{Aub}, and Schoen \cite{Sch} (see also \cite{LP} for an
overview of the Yamabe problem).

Implicitly, there has been much work on conformal parametrization of
two-dimensional piecewise flat manifolds, many of which start with a circle
packing on a region in $\mathbb{R}^{2}$ or a generalized circle packing on a
manifold. Thurston found a variational proof of Andreev's theorem
(\cite{Thurs} \cite{MR}) and conjectured that the Riemann mapping theorem
could be approximated by circle packing maps, which was soon proven to be true
\cite{RS} (see also \cite{Ste} for an overview of the theory). Another
direction for conformal parametrization appears in \cite{RW}, \cite{Luo1}, and
\cite{SSP}. These and other works produce a rich theory of conformal
geometries on surfaces and have led to many beautiful results about circle
packings and their generalizations. The theory developed in this paper unifies
several of these seemingly different notions of conformality to a more general
notion. It also allows an explicit computation of variations of angles which
allows one to glean geometric information. The geometric interpretation of the
variations of angles was known in some instances (e.g., \cite{He} \cite{G1}),
but the proofs were explicit computations, which made them difficult to extend
to more general cases. One of the main contributions of this paper is to show
how these computations may be done in a more simple, geometric way which
easily generalizes.

Existing literature on conformal parametrization of three-dimensional
piecewise flat manifolds is much more sparse. A notion was given by Cooper and
Rivin \cite{CR} which takes a sphere packing approach, and a rigidity result
was produced (see also \cite{Riv2} and \cite{G2}). However, this theory
requires that edge lengths come from a sphere packing, which is a major
restriction of the geometry even on a single tetrahedron. In \cite{G1}, the
author was able to show by explicit computation that the variations of angles
are related to certain areas and lengths of the piecewise flat manifold (in
actuality, one needs the additional structure of a metric as described below).
The theory developed in this paper generalizes this result to a general class
of three dimensional piecewise flat manifolds. This generalization allows a
geometric understanding of the variation of angles in a three-dimensional
piecewise flat manifold under conformal variations, and the space of conformal
variations is quite large and need not depend on the initial distribution of
the edge lengths (unlike \cite{G1}, where one must assume that the metric
comes form a sphere packing structure).

The variation formulas for the curvature allow one to introduce a theory of
functionals closely related to Riemannian functionals such as the
Einstein-Hilbert functional. In two dimensions, many of these functionals are
well studied, originally dating back to the work of Colin de Verdi\`{e}re
\cite{CdV}. In dimensions greater than three, the generalization of the
Einstein-Hilbert functional was suggested by Regge \cite{Reg} and has been
well studied both in the physics and mathematics communities (see \cite{Ham}
for an overview). Recently, the functional was used to provide a constructive
proof of Alexandrov's theorem that a surface with positive curvature is the
boundary of a polytope \cite{BI}. In this paper, we give a general
construction for two-dimensional functionals arising from a conformal
structure. We also consider variations of the Einstein-Hilbert-Regge
functional with respect to conformal variations. Variation of this functional
gives rise to notions of Ricci flat, Einstein, scalar zero, and constant
scalar curvature metrics on piecewise flat manifolds. Our structure allows one
to consider second variations of these functionals around fixed points, and
give rigidity conditions near a Ricci flat or scalar zero manifold. An
eventual goal is to prove theorems about the space of piecewise flat manifolds
analogous to ones on Riemannian manifolds, for instance \cite{Koi} \cite{Obat}.

Certain curvatures considered here have been shown to converge in measure to
scalar curvature measure by Cheeger-M\"{u}ller-Schrader in \cite{CMS}. The
proof in the general case does not appear to give the best convergence rate,
and it is an open problem what this best convergence rate may be. It would be
desirable to have a more precise control of the convergence and to prove a
convergence of Ricci curvatures or of Einstein manifolds on piecewise flat
spaces to Riemannian Einstein manifolds. Although the convergence result shows
convergence to scalar curvature measure, it has been suggested that these
curvatures are analogous to the curvature operator on a Riemannian manifold
\cite{Che}.

This paper is organized as follows. Section
\ref{section geometric structures and curvature} gives definitions of
geometric structures on piecewise flat manifolds in analogy to Riemannian
manifolds and shows the main theorems on variations of curvature functionals.
Section \ref{section angle variations} derives formulas for conformal
variations of angles. Section \ref{section curvature variations} translates
these results to variations of curvatures and curvature functionals. Section
\ref{section special cases} discusses some of the conformal structures already
studied and shows how they fit into the framework developed here. Finally,
Section \ref{section laplace} discusses discrete Laplacians, when they are
negative semidefinite operators, and how this implies convexity results for
curvature functionals and rigidity of certain metrics. The main theorems in
the paper are Theorems \ref{2D} and \ref{3D} on the variations of angles,
which could easily be applied to extend these results to the case of manifolds
with boundary, Theorems \ref{curv derivative 2D} and
\ref{theorem 3d curvature variation} on the variation of curvature, which give
analogues of the variation (\ref{scalar curvature var smooth}) of scalar
curvature under conformal deformation of a Riemannian metric, Theorems
\ref{2d variational} and \ref{theorem conformal 3d}\ on the variation of
curvature functionals, Theorems \ref{theorem 2d convexity} and
\ref{theorem 3d convexity} on convexity of curvature functionals, and Theorems
\ref{theorem 2d rigidity} and \ref{theorem 3d rigidity}\ on rigidity of zero
scalar curvature and Ricci flat manifolds.

\section{Geometric structures and
curvature\label{section geometric structures and curvature}}

\subsection{Metric structure}

We will consider certain analogues of Riemannian geometry. A Riemannian
manifold $\left(  M^{n},g\right)  $ is a smooth manifold $M$ together with a
symmetric, positive definite 2-tensor $g.$ A piecewise flat manifold is
defined similarly to the definitions in \cite{CMS}.

\begin{definition}
\label{definition piecewise flat}A triangulated manifold $\left(  M,T\right)
$ is a topological manifold $M$ together with a triangulation $T$ of $M.$ A
(triangulated) \emph{piecewise flat manifold} $\left(  M,T,\ell\right)  $ is a
triangulated manifold $\left(  M,T\right)  $ together with a function $\ell$
on the edges of the triangulation such that each simplex can be embedded in
Euclidean space as a (nondegenerate) Euclidean simplex with edge lengths
determined by $\ell$.
\end{definition}

Nondegeneracy can be expressed by the fact that all simplices have positive
volume. This condition can be realized as a function of the edge lengths using
the Cayley-Menger determinant formula for volumes of Euclidean simplices.

In this paper we will consider only closed, triangulated manifolds although
the definitions could be extended to more general spaces. We will describe
simplices as $\left\{  i,j,\ldots,k\right\}  ,$ where $i,j,k$ are natural
numbers. The length associated to an edge $\left\{  i,j\right\}  $ will be
denoted $\ell_{ij},$ area associated to $\left\{  i,j,k\right\}  $ will be
denoted $A_{ijk},$ and volume associated to $\left\{  i,j,k,\ell\right\}  $
will be denoted by $V_{ijk\ell}.$ Note that once lengths are assigned, area
and volume can be computed using, for instance, the Cayley-Menger determinant
formula. We will also use the notation $\gamma_{i,jk}$ to denote the angle at
vertex $i$ in triangle $\left\{  i,j,k\right\}  ,$ and sometimes drop $jk$
when it is clear which triangle we are considering. A dihedral angle at edge
$\left\{  i,j\right\}  $ in $\left\{  i,j,k,\ell\right\}  $ will be denoted
$\beta_{ij,k\ell}$ and $k\ell$ will be dropped when it is clear which
tetrahedron we are considering. In all of the following cases, the indices
after the comma will be dropped when the context is clear.

\begin{definition}
Let $V\left(  T\right)  $ denote the vertices of $T,$ let $E\left(  T\right)
$ denote the edges of $T,$ and let $E_{+}\left(  T\right)  $ denote the
directed edges in $T$ (there are two directed edges $\left(  i,j\right)  $ and
$\left(  j,i\right)  $ associated to each edge $\left\{  i,j\right\}  $). For
any of these vector spaces $X,$ let $X^{\ast}$ space of functions
$h:X\rightarrow\mathbb{R}$.
\end{definition}

Note that, for instance, if $d\in E_{+}\left(  T\right)  ^{\ast}$ then
$d=\sum_{\left(  i,j\right)  \in E_{+}}d\left(  i,j\right)  \phi_{ij}$, where
$\phi_{ij}$ is the standard basis of $E_{+}\left(  V\right)  ^{\ast}.$ We will
use $d_{ij}$ (as in Definition \ref{definition metric}) to denote either
$d\left(  i,j\right)  $ or the function $d\left(  i,j\right)  \phi_{ij},$ and
similarly with elements of $V\left(  T\right)  ^{\ast}$ and $E\left(
T\right)  ^{\ast}.$

\begin{remark}
We are implicitly assuming that the list of vertices determines the simplex
uniquely. This is just to make the notation more transparent. We could also
have indexed by simplices, such as $\ell_{\sigma^{2}},A_{\sigma^{3}}%
,\gamma_{\sigma^{0}\subset\sigma^{3}},$ etc. This latter notation is much
better if one wants to allow multiple simplices which share the same vertices.
\end{remark}

\begin{remark}
A piecewise flat manifold is a geometric manifold, in the sense that it can be
given a distance function in much the same way that a Riemannian manifold is
given a distance function, i.e., by minimizing over lengths of curves.
\end{remark}

The definition ensures that each simplex can be embedded isometrically in
Euclidean space. The image of vertex $i$ in Euclidean space will be denoted
$v_{i},$ the image of edge $\left\{  i,j\right\}  $ will be denoted
$v_{i}v_{j},$ etc.

Piecewise flat manifolds are not exactly the analogue of a Riemannian manifold
we will consider.

\begin{definition}
\label{definition metric}Let $\left(  M,T\right)  $ be a triangulated
manifold. A \emph{piecewise flat pre-metric} is an element $d\in E_{+}\left(
T\right)  ^{\ast}$ such that $\left(  M,T,\ell\right)  $ is a piecewise flat
manifold for the assignment $\ell_{ij}=d_{ij}+d_{ji}$ for every edge $\left\{
i,j\right\}  .$ A piecewise flat pre-metric $d$ is a \emph{metric} if for
every triangle $\left\{  i,j,k\right\}  $ in $T,$
\begin{equation}
d_{ij}^{2}+d_{jk}^{2}+d_{ki}^{2}=d_{ji}^{2}+d_{ik}^{2}+d_{kj}^{2}%
.\label{d-condition}%
\end{equation}
A \emph{piecewise flat, metrized manifold} $\left(  M,T,d\right)  $ is a
triangulated manifold $\left(  M,T\right)  $ with metric $d.$
\end{definition}

For future use, we define the space of piecewise flat metrics on $\left(
M,T\right)  .$

\begin{definition}
\label{definition met}Define the space $\mathfrak{met}\left(  M,T\right)  $ to
be
\[
\mathfrak{met}\left(  M,T\right)  \mathfrak{=}\left\{  d\in E_{+}\left(
T\right)  ^{\ast}:\left(  M,T,d\right)  \text{ is a piecewise flat, metrized
manifold}\right\}  .
\]

\end{definition}

As shown in \cite{G3}, condition (\ref{d-condition}) ensures that every
simplex has a geometric center and a geometric dual which intersects the
simplex orthogonally at the center. This dual is constructed from centers.
Given a simplex embedded into space as $\left\{  v_{1},v_{2},\ldots
,v_{n}\right\}  ,$ we have a center point to the simplex given by
$c_{123\cdots n}.$ This point can be projected onto the $\left(  n-1\right)
$-dimensional simplices and successively projected onto all simplices, giving
centers $c_{ij\cdots k}$ for all subsets of $\left\{  1,2,\ldots,n\right\}  .$
The centers can be constructed inductively by starting with centers of edges
at a point $c_{ij}$ which is a (signed) distance $d_{ij}$ from vertex $i$ and
$d_{ji}$ from vertex $j.$ Then one considers orthogonal lines through the
centers, and condition (\ref{d-condition}) ensures that in each triangle,
there is a single point where these three lines intersect, giving a center for
the triangle. The construction may be continued for all dimensions, as
described in \cite{G3}. 

For simplicity, let's restrict to $n\leq4.$ We will denote the signed distance
between $c_{1234}$ and $c_{ijk}$ by $h_{ijk,\ell}$ and the signed distance
between $c_{ijk}$ and $c_{ij}$ by $h_{ij,k}.$ The sign is gotten by the
following convention. If $c_{1234}$ is on the same side of the plane defined
by $v_{i}v_{j}v_{k}$ as the tetrahedron $v_{1}v_{2}v_{3}v_{4},$ then
$h_{ijk,\ell}$ is positive, otherwise it is negative (or zero if the point is
on that plane). Similarly, if $c_{ijk}$ is on the same side of the line
defined by $v_{i}v_{j}$ as $v_{i}v_{j}v_{k}$ within that plane, then
$h_{ij,k}$ is positive. Since it is clear that $h_{i,j}=d_{ij},$ we will not
use the former. The side $v_{i}v_{j}$ is divided into a segment containing
$v_{i}$ of length $d_{ij}$ and a segment containing $v_{j}$ of length $d_{ji}$
such that $\ell_{ij}=d_{ij}+d_{ji}.$ It is easy to deduce that $h_{ij,k}$ and
$h_{ijk,\ell}$ can be computed by
\[
h_{ij,k}=\frac{d_{ik}-d_{ij}\cos\gamma_{i,jk}}{\sin\gamma_{i,jk}}%
\]
and
\[
h_{ijk,\ell}=\frac{h_{ij,\ell}-h_{ij,k}\cos\beta_{ij,k\ell}}{\sin
\beta_{ij,k\ell}}.
\]
See \cite{G3} or \cite{BI} for a proof. Importantly, these quantities work for
negative values of the $d$'s and $h$'s. We will also consider the dual area
$A_{ij,k\ell}$ of the edge $\left\{  i,j\right\}  $ in tetrahedron $\left\{
i,j,k,\ell\right\}  ,$ which is the signed area of the planar quadrilateral
$c_{1234}c_{ijk}c_{ij}c_{ij\ell},$ where $i,j,k,\ell$ are distinct. The area
is equal to
\[
A_{ij,k\ell}=\frac{1}{2}\left(  h_{ij,k}h_{ijk,\ell}+h_{ij,\ell}h_{ij\ell
,k}\right)  .
\]

These definitions of centers within a simplex induce a definition of geometric
duals on a triangulation (see \cite{G3} for details). In particular, we will
need the lengths or areas of duals of edges, defined in two and three
dimensions as follows.

\begin{definition}
\label{definition 2d dual length}Let $\left(  M^{2},T,d\right)  $ be a
piecewise flat, metrized manifold of dimension $2.$ Then edge $\left\{
i,j\right\}  $ is the boundary of two triangles, say $\left\{  i,j,k\right\}
$ and $\left\{  i,j,\ell\right\}  .$ The dual length $\ell_{ij}^{\ast}$ is
defined as
\[
\ell_{ij}^{\ast}=h_{ij,k}+h_{ij,\ell}.
\]

\end{definition}

Note that the two triangles can be embedded in the Euclidean plane together,
and $\ell_{ij}^{\ast}$ is the signed distance between the centers of the two triangles.

\begin{definition}
\label{definition 3d dual length}Let $\left(  M^{3},T,d\right)  $ be a
piecewise flat, metrized manifold of dimension $3.$ Then the dual length
$\ell_{ij}^{\ast}$ (which is technically an area) is defined as
\begin{align*}
\ell_{ij}^{\ast} &  =\sum_{k,\ell}A_{ij,k\ell}\\
&  =\sum_{k,\ell}\frac{1}{2}\left(  h_{ij,k}h_{ijk,\ell}+h_{ij,\ell}%
h_{ij\ell,k}\right)  ,
\end{align*}
where the sum is over all tetrahedra containing the edge $\left\{
i,j\right\}  .$
\end{definition}

\begin{notation}
Most sums in this paper will be with respect to simplices, so a sum such as
the one in Definition \ref{definition 3d dual length} means the sum over all
tetrahedra $\left\{  i,j,k,\ell\right\}  $ containing the edge $\left\{
i,j\right\}  ,$ not the sum over all values of $k$ and $\ell$ (which would
give twice the aforementioned sum). 
\end{notation}

The dual length is the area of a (generalized) polygon which intersects the
edges orthogonally at their centers.

\begin{remark}
We specifically did not use the word \emph{Riemannian }because it is not
entirely clear what Riemannian should mean. Natural guesses would be that
$d_{ij}>0$ for all directed edges $\left(  i,j\right)  $ or that all dual
volumes are positive. However, we chose not to make such a distinction in this paper.
\end{remark}

\subsection{Curvature}

In this section we define curvatures of piecewise flat metrized manifolds,
many of which are the same as those for piecewise flat manifolds described by
Regge \cite{Reg} and Cheeger-M\"{u}ller-Schrader \cite{CMS}. Generally,
curvature on a piecewise flat manifold of dimension $n$ is considered to be
concentrated on codimension $2$ simplices, and the curvature at $\sigma$ is
equal to the dihedral angle deficit from $2\pi$ multiplied by the volume of
$\sigma,$ possibly with a normalization. Cheeger-M\"{u}ller-Schrader
\cite{CMS} show that, under appropriate convergence of the triangulations,
such a curvature converges in measure to scalar curvature measure $RdV.$ (In
fact, Cheeger-M\"{u}ller-Schrader prove a much more general result for all
Lipschitz-Killing curvatures, but we will only consider scalar curvature.) We
first define curvature for piecewise flat manifolds in dimension $2,$ which is
concentrated at vertices.

\begin{definition}
\label{definition 2d curvature}Let $\left(  M,T,\ell\right)  $ be a
two-dimensional piecewise flat manifold. Then the curvature $K_{i}$ at a
vertex $i$ is equal to
\[
K_{i}=2\pi-\sum_{j,k}\gamma_{i,jk},
\]
where $\gamma_{i}$ are the interior angles of the triangles at vertex $i.$
\end{definition}

Angles can be calculated from edge lengths using the law of cosines. Note that
in two dimensions, curvature satisfies a discrete Gauss-Bonnet equation,
\[
\sum_{i}K_{i}=2\pi\chi\left(  M\right)  ,
\]
where $\chi$ is the Euler characteristic.

In dimension $3,$ the curvature is concentrated at edges.

\begin{definition}
\label{definition 3d edge curvature}Let $\left(  M,T,\ell\right)  $ be a
three-dimensional piecewise flat manifold. Then the edge curvature $K_{ij}$ is%
\[
K_{ij}=\left(  2\pi-\sum_{k,\ell}\beta_{ij,k\ell}\right)  \ell_{ij}.
\]

\end{definition}

The dihedral angles can be computed as a function of edge lengths using the
Euclidean cosine law to get the face angles, and then using the spherical
cosine law to related the face angles to a dihedral angle. 

There is an interpretation of $K_{ij}/\ell_{ij}$ in terms of deficits of
parallel translations around the \textquotedblleft bone\textquotedblright%
\ $\left\{  i,j\right\}  $. (See \cite{Reg} for details.) For this reason, one
may think of $K_{ij}/\ell_{ij}$ as some sort of analogue of sectional
curvature or curvature operator (see \cite{Che}).

The fact that curvature is concentrated at edges often makes it difficult to
compare curvatures with functions, which are naturally defined at vertices.
For this reason, we will try move these curvatures to curvature functions
based at vertices.

In the smooth case, the scalar curvature has interesting variation formulas.
For instance, we may consider the Einstein-Hilbert functional,
\[
\mathcal{EH}\left(  M,g\right)  =\int_{M}R_{g}dV_{g},
\]
where $R_{g}$ is the scalar curvature and $dV_{g}$ is the Riemannian volume
measure. Note that if $n=2,$ then the Gauss-Bonnet theorem says that
$\mathcal{EH}\left(  M,g\right)  =2\pi\chi\left(  M\right)  ,$ but otherwise
this functional is an interesting one geometrically. A well-known calculation
(see, for instance, \cite{Bes}) shows that if we consider variations of the
Riemannian metric $\delta g=h$ on $M^{n},$ then
\begin{equation}
\delta\mathcal{EH}\left(  M,g\right)  \left[  h\right]  =\int_{M}g\left(
E,h\right)  dV,\label{EH first variation general}%
\end{equation}
where $E=R_{ij}-\frac{1}{2}Rg_{ij}$ is the Einstein tensor. It follows that
critical points of this functional satisfy
\begin{equation}
R_{ij}-\frac{1}{2}Rg_{ij}=0.\label{Einstein zero eqn}%
\end{equation}
Taking the trace of this equation with respect to the metric, we see that, if
$n\neq2$, this implies that
\begin{equation}
R_{ij}=0,\label{ricci flat eqn}%
\end{equation}
which is the Einstein or Ricci-flat equation. It also makes sense to consider
either the constrained problem where volume is equal to one, or to consider
the normalized functional
\[
\frac{\mathcal{EH}\left(  M^{n},g\right)  }{\mathcal{V}\left(  M^{n},g\right)
^{\left(  n-2\right)  /n}},
\]
where $\mathcal{V}$ is the volume. In both cases we find that critical points
under a conformal variation correspond to metrics satisfying
\begin{equation}
R_{ij}=\lambda g_{ij}\label{einstein eqn}%
\end{equation}
for a constant $\lambda.$ Taking the trace and integrating, we see that
\begin{equation}
\lambda=\frac{1}{n}\frac{\mathcal{EH}\left(  M^{n},g\right)  }{\mathcal{V}%
\left(  M^{n},g\right)  }.\label{lambda formula}%
\end{equation}

We now consider Regge's analogue to the Einstein-Hilbert functional on
three-dimensional piecewise flat manifolds.

\begin{definition}
\label{definition EHR}If $\left(  M^{3},T,\ell\right)  $ is a
three-dimensional piecewise flat manifold, the \emph{Einstein-Hilbert-Regge
functional} $\mathcal{EHR}$ is%
\begin{equation}
\mathcal{EHR}\left(  M,T,\ell\right)  =\sum_{i,j}K_{ij},
\end{equation}
where the sum is over all edges $\left\{  i,j\right\}  \in E\left(  T\right)
.$
\end{definition}

The analogue of the first variation formula (\ref{EH first variation general})
is
\begin{equation}
\frac{\partial}{\partial\ell_{ij}}\mathcal{EHR}\left(  M,T,\ell\right)
=2\pi-\sum_{k,\ell}\beta_{ij,k\ell}.\label{regge variation}%
\end{equation}
This was proven by Regge \cite{Reg} and follows immediately from the
Schl\"{a}fli formula (see \cite{Mil}). By analogy with the smooth case, we
define the following.

\begin{definition}
A piecewise flat manifold $\left(  M^{3},T,\ell\right)  $ is \emph{Ricci flat}
if
\[
K_{ij}=0
\]
for all edges $\left\{  i,j\right\}  .$ It is \emph{Einstein} with Einstein
constant $\lambda\in\mathbb{R}$ if
\begin{equation}
K_{ij}=\lambda\ell_{ij}\frac{\partial\mathcal{V}}{\partial\ell_{ij}%
},\label{discrete einstein eqn}%
\end{equation}
for all edges $\left\{  i,j\right\}  ,$ where
\[
\mathcal{V}\left(  M,T,\ell\right)  \mathcal{=}\sum_{i,j,k,\ell}V_{ijk\ell}%
\]
is the total volume.
\end{definition}

The term on the left of (\ref{discrete einstein eqn}) can be made more
explicit. Note that
\begin{align*}
3\mathcal{V}\left(  M,T,\ell\right)   &  =\left.  \frac{d}{da}\right\vert
_{a=1}a^{3}\mathcal{V}\left(  M,T,\ell\right)  \\
&  =\left.  \frac{d}{da}\right\vert _{a=1}\mathcal{V}\left(  M,T,a~\ell
\right)  \\
&  =\sum_{i,j}\ell_{ij}\frac{\partial\mathcal{V}}{\partial\ell_{ij}},
\end{align*}
so
\[
\lambda=\frac{\mathcal{EHR}\left(  M,T,\ell\right)  }{3\mathcal{V}\left(
M,T,\ell\right)  },
\]
analogous to the smooth formula (\ref{lambda formula}). Furthermore, we can
explicitly compute for any tetrahedron $\left\{  i,j,k,\ell\right\}  $ that%
\begin{equation}
\frac{\partial V_{ijk\ell}}{\partial\ell_{ij}}=\frac{1}{6}\ell_{ij}\ell
_{k\ell}\cot\beta_{k\ell,ij}.\label{variation fmla volume length}%
\end{equation}
For brevity, we omit the proof of (\ref{variation fmla volume length}) since
we will not use it. However, it can be proven by a direct computation of the
derivatives of volume and of the dihedral angle.

As in the smooth case, studying the Einstein equation is quite difficult.
Progress can be made by considering only certain variations of the metric. If
one takes $\delta g=fg$ for a function $f,$ we have a conformal variation.
Under conformal variations, the scalar curvature satisfies%
\begin{equation}
\delta R\left[  fg\right]  =\left(  1-n\right)  \triangle
f-Rf.\label{scalar curvature var smooth}%
\end{equation}
Since, under this variation, $\delta dV=\frac{n}{2}fdV,$ the variation of
$\mathcal{EH}$ under a conformal variation is
\begin{align*}
\delta\mathcal{EH}\left(  M,g\right)  \left[  fg\right]   &  =\int\left[
\left(  1-n\right)  \triangle f+\left(  \frac{n}{2}-1\right)  Rf\right]  dV\\
&  =\left(  \frac{n}{2}-1\right)  \int RfdV.
\end{align*}
In particular, $\left(  \frac{n}{2}-1\right)  R$ is the gradient of
$\mathcal{EH}$ with respect to the $L_{2}\left(  M,dV\right)  $ inner product.
We see that critical points of the functional under conformal variations
correspond to when the scalar curvature is zero. Note that if we either (a)
restrict to metrics with volume 1 or (b) normalize the functional, then we get
constant scalar curvature metrics as critical points. The second variation of
$\mathcal{EH}$ can be calculated from (\ref{scalar curvature var smooth}) to
be
\begin{align*}
\delta^{2}\mathcal{EH}\left(  M,g\right)  \left[  fg,fg\right]   &  =\left(
\frac{n}{2}-1\right)  \int_{M}\left[  \left(  1-n\right)  f\triangle f+\left(
\frac{n}{2}-1\right)  Rf^{2}\right]  dV\\
&  =\left(  \frac{n}{2}-1\right)  \int_{M}\left[  \left(  n-1\right)
\left\vert \nabla f\right\vert ^{2}+\left(  \frac{n}{2}-1\right)
Rf^{2}\right]  dV.
\end{align*}
The second variation can be used to check to see if critical points are rigid,
i.e., if there is a family of deformations of critical metrics.

The discrete formulation is motivated by the work of Cooper and Rivin
\cite{CR}, who looked at the sphere packing case. The goal is to formulate a
conformal theory in the piecewise flat setting which allows simple variation
formulas as in the smooth setting. First, we define the scalar curvature.

\begin{definition}
\label{definition 3d scalar curvature}The \emph{scalar curvature} $K$ of a
three-dimensional piecewise flat, metrized manifold $\left(  M^{3},T,d\right)
$ is the function on the vertices defined by
\[
K_{i}=\sum_{j}\left(  2\pi-\sum_{k,\ell}\beta_{ij,k\ell}\right)  d_{ij}.
\]

\end{definition}

This definition is much more general than the one in \cite{CR}, but restricts
to almost the same definition in the case of sphere packing (see Section
\ref{section special cases} for the details). This curvature is in many ways
analogous to the scalar curvature measure $RdV$ on a Riemannian manifold. Note
that, unlike the edge curvatures $K_{ij},$ this curvature depends on the
metric, not only the piecewise flat manifold. We also note the following
important fact.

\begin{proposition}
If $\left(  M^{3},T,d\right)  $ is a three-dimensional piecewise flat,
metrized manifold, then the Einstein-Hilbert-Regge functional can be written
\[
\mathcal{EHR}\left(  M,T,\ell\left(  d\right)  \right)  =\sum_{i}K_{i}.
\]

\end{proposition}

\begin{proof}
Simply do the sum and recall that $d_{ij}+d_{ji}=\ell_{ij}.$
\end{proof}

Now let us define conformal structure. The motivation for the definition will
be seen in Theorems \ref{2d variational} and \ref{theorem conformal 3d}, and
we will see some examples in Section \ref{section special cases}. The reader
may want to recall Definition \ref{definition met}.

\begin{definition}
A \emph{conformal structure} $\mathcal{C}\left(  M,T,U\right)  $ on a
triangulated manifold $\left(  M,T\right)  $ on an open set $U\subset V\left(
T\right)  ^{\ast}$ is a smooth map%
\[
\mathcal{C}\left(  M,T,U\right)  :U\rightarrow\mathfrak{met}\left(
M,T\right)
\]
such that if $d=\mathcal{C}\left(  M,T,U\right)  \left[  f\right]  $ then for
each $\left(  i,j\right)  \in E_{+}$ and $k\in V$,
\[
\frac{\partial\ell_{ij}}{\partial f_{i}}=d_{ij}%
\]
and
\[
\frac{\partial d_{ij}}{\partial f_{k}}=0
\]
if $k\neq i$ and $k\neq j$.
\end{definition}

\begin{notation}
Often we will suppress the $U$ and simply refer to the domain of the conformal
structure $\mathcal{C}\left(  M,T\right)  .$
\end{notation}

We can also define a conformal variation.

\begin{definition}
A \emph{conformal variation} of a piecewise flat, metrized manifold $\left(
M,T,\bar{d}\right)  $ is a smooth curve $f:\left(  -\varepsilon,\varepsilon
\right)  \rightarrow V\left(  T\right)  ^{\ast}$ such that there exists a
conformal structure $\mathcal{C}\left(  M,T,U\right)  $ with $f\left(
-\varepsilon,\varepsilon\right)  \subset U$ and $f\left(  0\right)  =\bar{d}.$
We call such a conformal structure an \emph{extension} of the conformal variation.
\end{definition}

An important point is that if we have a conformal structure or conformal
variation, quantities such as $\frac{\partial\ell_{ij}}{\partial f_{j}}$ make
sense. We will usually try to make statements in terms of $\frac{df_{i}}{dt}$
in order to reveal the appearance of discrete Laplacians, however sometimes it
will be more convenient to express terms as partial derivatives. We note that
a conformal variation is essentially independent of the extension in the
following sense.

\begin{proposition}
Under a conformal variation $f\left(  t\right)  $ of $\left(  M,T,\bar
{d}\right)  $, we have at $t=0$ that
\[
\frac{d}{dt}\ell_{ij}=\bar{d}_{ij}\frac{df_{i}}{dt}+\bar{d}_{ji}\frac{df_{j}%
}{dt}.
\]
In particular, at $t=0,$ for a given $\frac{df}{dt}\left(  0\right)  ,$ the
variation of the length is independent of the extension.
\end{proposition}

\begin{proof}
From the definition of conformal structure, we have
\begin{align*}
\frac{d}{dt}\ell_{ij} &  =\frac{\partial\ell_{ij}}{\partial f_{i}}\frac
{df_{i}}{dt}+\frac{\partial\ell_{ij}}{\partial f_{j}}\frac{df_{j}}{dt}\\
&  =d_{ij}\frac{df_{i}}{dt}+d_{ji}\frac{df_{j}}{dt}.
\end{align*}

\end{proof}

\begin{notation}
In the sequel, when we suppose a conformal variation, it will be understood
that quantities such as $\frac{df_{i}}{dt}$ are evaluated at $t=0,$ though not stated.
\end{notation}

There are often more than one extension to a conformal variation. For
instance, for a triangle $\left\{  1,2,3\right\}  ,$ we may extend the metric
defined by $d_{ij}=\frac{1}{2}$ for all $\left(  i,j\right)  \in E_{+}$ to
several families where $f_{i}\left(  t\right)  =tx_{i},$ such as
\[
d_{ij}\left(  t\right)  =\frac{1}{2}\exp\left(  tx_{i}\right)  ,
\]
which corresponds to a circle packing conformal structure (see Section
\ref{section sphere packing}), and
\[
d_{ij}\left(  t\right)  =\frac{1}{2}\exp\left(  \frac{t}{2}\left(  x_{i}%
+x_{j}\right)  \right)  ,
\]
which corresponds to a perpendicular bisector conformal structure (see Section
\ref{section perpendicular bisector}).

\begin{remark}
Often a conformal structure will be generated from a base metric, much the
same way a conformal class on a Riemannian manifolds can be described as the
equivalence class of metrics $e^{f}g_{0},$ where $f$ is a function on the
manifold and $g_{0}$ is the base Riemannian metric. However, we have not
defined it thus and, in general, one must be careful how the structures are
defined if one wishes to partition all piecewise flat manifolds into conformal
classes. We do not attempt this here, though there is a straightforward way to
do this for perpendicular bisector conformal structures seen in Section
\ref{section perpendicular bisector}.
\end{remark}

In two dimensions, the fact that curvatures arise from conformal variations of
a functional is not obvious, but can be proven.

\begin{theorem}
\label{2d variational}Fix a conformal structure $\mathcal{C}\left(
M^{2},T,U\right)  $ on a two dimensional triangulated manifold and suppose
that $U$ is simply connected. Then there is a functional $F:U\rightarrow
\mathbb{R}$ such that
\[
\frac{\partial F}{\partial f_{i}}=K_{i}%
\]
for each $i\in V\left(  T\right)  .$ Furthermore, the second variation of the
functional under a conformal variation $f\left(  t\right)  $ can be expressed
as
\begin{equation}
\frac{d^{2}F}{dt^{2}}=\frac{1}{2}\sum_{i,j}\frac{\ell_{ij}^{\ast}}{\ell_{ij}%
}\left(  \frac{df_{j}}{dt}-\frac{df_{i}}{dt}\right)  ^{2}+K_{i}\frac
{d^{2}f_{i}}{dt^{2}}.\label{2d functional deriv laplace}%
\end{equation}

\end{theorem}

This sort of formulation of the prescribed curvature problem in a variational
framework has been studied by many people. See, for instance, \cite{DGL}
\cite{Riv1} \cite{CdV} \cite{CL} \cite{BS1} \cite{Spr} \cite{Guo}. However, no source to
date has unified the theorem in the way of Theorem \ref{2d variational}.

In three dimensions, the Einstein-Hilbert-Regge functional is a natural one to consider.

\begin{theorem}
\label{theorem conformal 3d}For any conformal variation $f\left(  t\right)  $
of a three dimensional, piecewise flat, metrized manifold $\left(  M,T,\bar
{d}\right)  ,$ we have,
\begin{align}
\frac{d}{dt}\mathcal{EHR}\left(  M,T,\ell\left(  f\left(  t\right)  \right)
\right)   &  =\sum_{i}K_{i}\frac{df_{i}}{dt}\label{EHR first var}\\
\frac{d^{2}}{dt^{2}}\mathcal{EHR}\left(  M,T,\ell\left(  f\left(  t\right)
\right)  \right)   &  =\sum_{i}\sum_{j\neq i}\left(  \frac{\ell_{ij}^{\ast}%
}{\ell_{ij}}-\frac{q_{ij}}{2\ell_{ij}}K_{ij}\right)  \left(  \frac{df_{j}}%
{dt}-\frac{df_{i}}{dt}\right)  ^{2}\label{EHR second var gen}\\
&  \;\;\;\;\;\;\;+\sum_{i}K_{i}\left[  \left(  \frac{df_{i}}{dt}\right)
^{2}+\frac{d^{2}f_{i}}{dt^{2}}\right]  ,\nonumber
\end{align}
where%
\[
q_{ij}=\frac{\partial d_{ij}}{\partial f_{j}}=\frac{\partial d_{ji}}{\partial
f_{i}}.
\]
Thus a critical point of $\mathcal{EHR}$ corresponds to when $K_{i}=0$ for all
$i$ and at a critical metric,
\begin{equation}
\frac{d^{2}}{dt^{2}}\mathcal{EHR}\left(  M,T,\ell\left(  f\left(  t\right)
\right)  \right)  =\sum_{i}\sum_{j\neq i}\left(  \frac{\ell_{ij}^{\ast}}%
{\ell_{ij}}-\frac{q_{ij}}{2\ell_{ij}}K_{ij}\right)  \left(  \frac{df_{j}}%
{dt}-\frac{df_{i}}{dt}\right)  ^{2}.\label{EHR second var zero scalar}%
\end{equation}
Furthermore, at a critical point for the general $\mathcal{EHR}\left(
M,T,\ell\right)  ,$ we must have $K_{ij}=0$ (see (\ref{regge variation})), and
hence here we have
\begin{equation}
\frac{d^{2}}{dt^{2}}\mathcal{EHR}\left(  M,T,\ell\left(  f\left(  t\right)
\right)  \right)  =\sum_{i}\sum_{j\neq i}\frac{\ell_{ij}^{\ast}}{\ell_{ij}%
}\left(  \frac{df_{j}}{dt}-\frac{df_{i}}{dt}\right)  ^{2}%
.\label{EHR second var ricci flat}%
\end{equation}

\end{theorem}

Theorem \ref{theorem conformal 3d} motivates the definition of constant scalar
curvature metrics, as seen from the following.

\begin{corollary}
\label{corollary volume variation}For any conformal structure of $\left(
M,T\right)  ,$ we have
\begin{align*}
\frac{\partial}{\partial f_{i}}\mathcal{EHR}\left(  M,T,\ell\left(  f\right)
\right)   &  =K_{i},\\
\frac{\partial}{\partial f_{i}}\mathcal{V}\left(  M,T,\ell\left(  f\right)
\right)   &  =V_{i},
\end{align*}
where
\[
V_{i}=\frac{1}{3}\sum_{j,k,\ell}h_{ijk,\ell}A_{ijk}.
\]

\end{corollary}

The proof will be given in Section \ref{section 3d curv vars}. Corollary
\ref{corollary volume variation} motivates the following definition.

\begin{definition}
A three-dimensional piecewise flat, metrized manifold $\left(  M^{3}%
,T,d\right)  $ is has \emph{constant scalar curvature} $\lambda$ if
\[
K_{i}=\lambda V_{i}%
\]
for all vertices $i.$
\end{definition}

It is not hard to see that
\[
\sum_{i}V_{i}=3\mathcal{V}.
\]
Summing both sides of the constant scalar curvature equation, we see that
\[
\lambda=\frac{1}{3}\frac{\mathcal{EHR}\left(  M,T,\ell\right)  }%
{\mathcal{V}\left(  M,T,\ell\right)  }.
\]
Note that
\[
V_{i}=\frac{\partial\mathcal{V}}{\partial f_{i}}=\sum_{j}\frac{\partial
\mathcal{V}}{\partial\ell_{ij}}d_{ij},
\]
and so on an Einstein manifold, which would satisfy
\[
\frac{\partial}{\partial\ell_{ij}}\mathcal{EHR=}\frac{K_{ij}}{\ell_{ij}%
}\mathcal{=\lambda}\frac{\partial\mathcal{V}}{\partial\ell_{ij}}%
\]
for each edge, we have that
\[
K_{i}=\sum_{j}\frac{K_{ij}}{\ell_{ij}}d_{ij}=\lambda\sum_{j}\frac
{\partial\mathcal{V}}{\partial\ell_{ij}}d_{ij}=\lambda V_{i}.
\]
We have just proved the following.

\begin{proposition}
If $\left(  M^{3},T,d\right)  $ is a three-dimensional piecewise flat,
metrized manifold which is Einstein, then it has constant scalar curvature.
\end{proposition}

There is a second variation formula for conformal variations of $\mathcal{EHR}%
/\mathcal{V}^{1/3}$ at Einstein manifolds, but for brevity we omit it since it
requires the calculation of $\frac{\partial V_{i}}{\partial f_{j}}.$ With the
results from this paper, it is straightforward to calculate these derivatives.

\section{Variations of angles\label{section angle variations}}

In the rest of this paper, we will use $\delta$ to denote the differential.

\subsection{Two dimensions}%

\begin{figure}
[tb]
\begin{center}
\includegraphics[
natheight=3.858800in,
natwidth=3.245300in,
height=3.7999in,
width=3.1988in
]%
{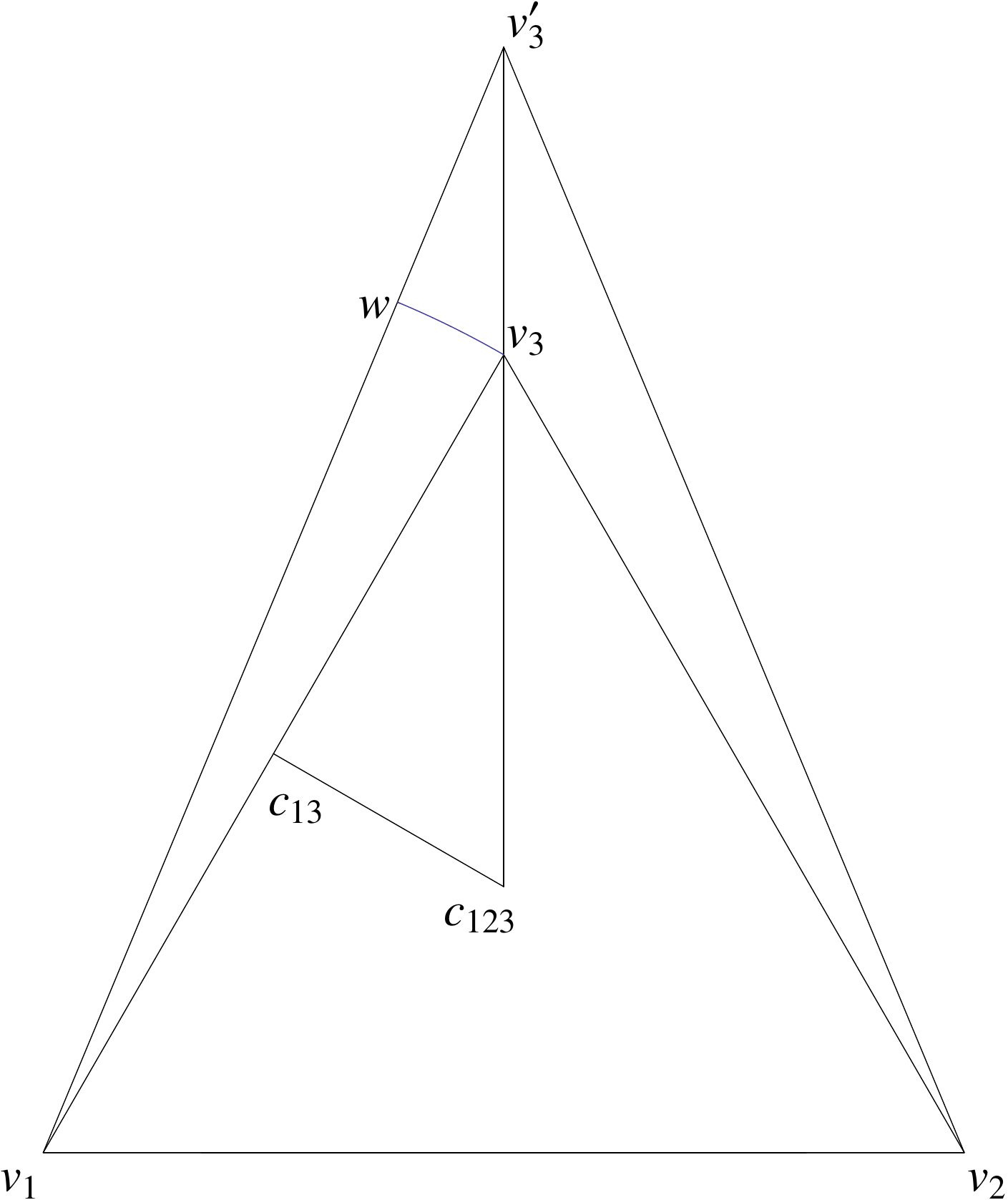}%
\caption{Variation of a triangle.}%
\label{2dpic1}%
\end{center}
\end{figure}
\begin{figure}
[tb]
\begin{center}
\includegraphics[
natheight=3.199600in,
natwidth=5.363200in,
height=3.2287in,
width=5.3923in
]%
{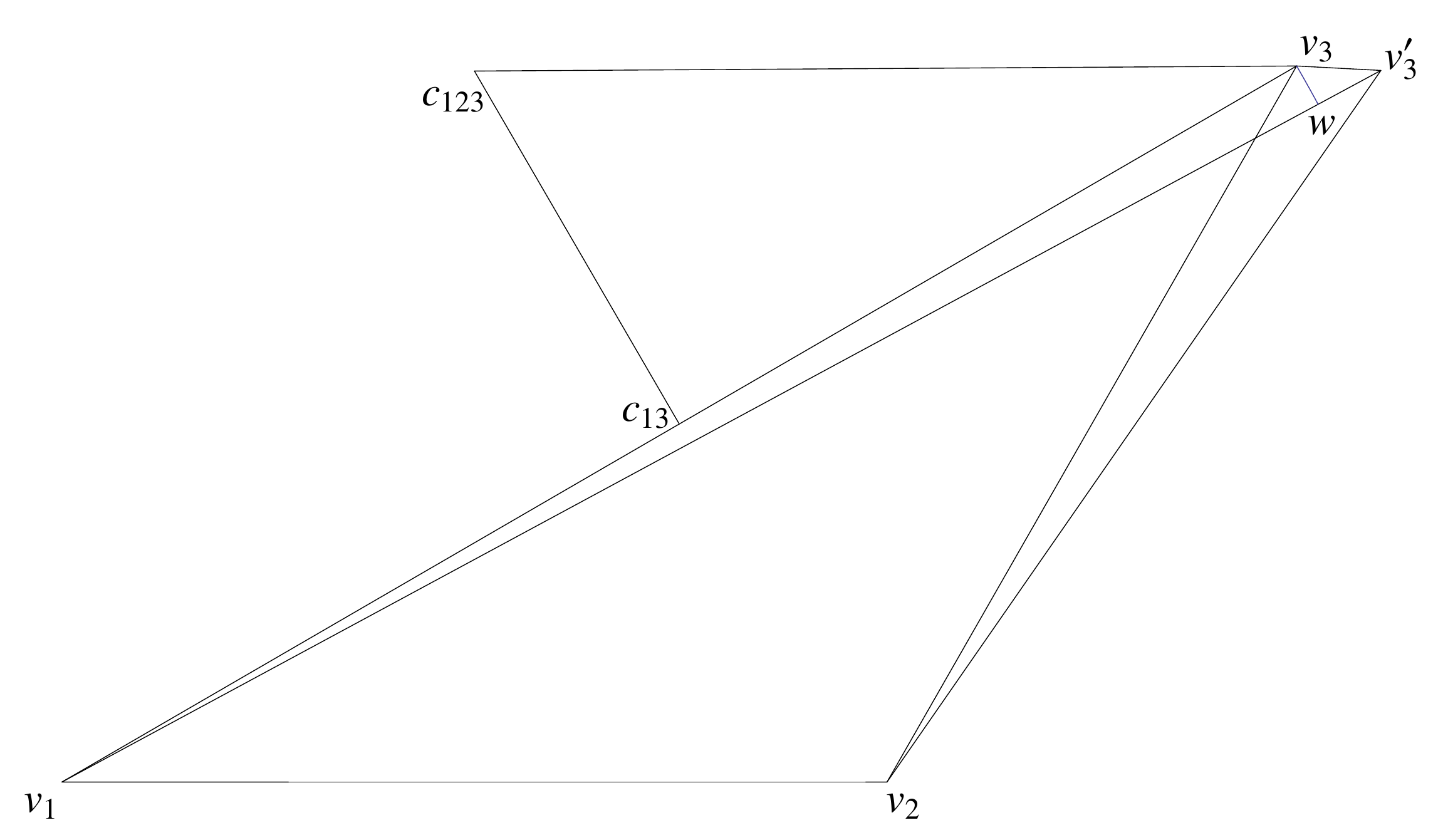}%
\caption{Variation of a triangle where $c_{123}$ is outside the triangle.}%
\label{2dpic2}%
\end{center}
\end{figure}
In this section we will compute the derivative of an angle under a certain
variation of lengths. Consider the Euclidean triangle determined by lengths
$\left(  \ell_{12},\ell_{13},\ell_{23}\right)  $ with vertices $\left\{
v_{1},v_{2},v_{3}\right\}  $ and also the triangle determined by lengths
$\left(  \ell_{12},\ell_{13}+\delta\ell_{13},\ell_{23}+\delta\ell_{23}\right)
$, say with vertices $\left\{  v_{1},v_{2},v_{3}^{\prime}\right\}  $, under
the important assumptions:%
\begin{align}
\delta\ell_{12} &  =0,\label{lengthvar1}\\
\delta\ell_{13} &  =d_{31}\delta f_{3},\label{lengthvar2}\\
\delta\ell_{23} &  =d_{32}\delta f_{3},\label{lengthvar3}%
\end{align}
where $d$ is a metric on $\left\{  1,2,3\right\}  $ inducing lengths $\ell.$

Draw the arc representing $\ell_{14}~\delta\gamma_{1,24},$ which goes through
vertex $v_{3}$ and intersects the segment $v_{1}v_{3}^{\prime}.$ Call this
edge $E.$ It has endpoints $v_{3}$ and $w.$ See Figures \ref{2dpic1} and
\ref{2dpic2} for cases when the center $c_{123}$ is inside $v_{1}v_{2}v_{3}$
and when it is outside $v_{1}v_{2}v_{3}$.

\begin{proposition}
\label{straightline}The points $c_{123},$ $v_{3},$ and $v_{3}^{\prime}$ lie on
a line. I.e., $\delta v_{3}$ is parallel to $v_{3}-c_{123}.$
\end{proposition}

\begin{proof}
Notice that
\[
\delta\left(  \ell_{13}^{2}\right)  =\delta\left[  \left(  v_{3}-v_{1}\right)
\cdot\left(  v_{3}-v_{1}\right)  \right]  =2\left(  v_{3}-v_{1}\right)
\cdot\delta v_{3}%
\]
but also
\[
\delta\left(  \ell_{13}^{2}\right)  =2\ell_{13}d_{31}\delta f_{3},
\]
so
\[
\delta v_{3}\cdot\frac{v_{3}-v_{1}}{\ell_{13}}=d_{31}\delta f_{3}.
\]
Similarly,
\[
\delta v_{3}\cdot\frac{v_{3}-v_{2}}{\ell_{23}}=d_{32}\delta f_{3}.
\]
Similarly, the vector $v_{3}-c_{123}$ satisfies
\begin{align*}
\left(  v_{3}-c_{123}\right)  \cdot\frac{v_{3}}{\ell_{13}} &  =d_{31},\\
\left(  v_{3}-c_{123}\right)  \cdot\frac{v_{3}-v_{2}}{\ell_{23}} &  =d_{32}%
\end{align*}
and so we see that $\delta v_{3}=\left(  v_{3}-c_{123}\right)  \delta f_{3}.$
\end{proof}

This is essentially the same proof given by Thurston \cite{Thurs} and
Marden-Rodin \cite{MR}.

Consider the triangle $v_{3}wv_{3^{\prime}}.$ This is a right triangle with
right angle at $w$ since $v_{1}w$ is a radius of the circle containing $E.$
Since the angle of $E$ with $v_{1}v_{3}$ is also a right angle, together with
Proposition \ref{straightline}, it follows that $v_{3}wv_{3^{\prime}}$ is
similar to the right triangle $c_{123}c_{13}v_{3}$. Using the similar
triangles, we get that%
\begin{equation}
\frac{\ell_{13}\left(  \delta\gamma_{1,23}\right)  }{\delta\ell_{13}}%
=\frac{h_{13,2}}{d_{31}}.\label{key 2d relation}%
\end{equation}
So
\[
\frac{\partial\gamma_{1,23}}{\partial f_{3}}=\frac{h_{13,2}}{\ell_{13}}.
\]
This leads to the following theorem.

\begin{theorem}
\label{2D}For variations of the lengths of a Euclidean triangle of the type
\ref{lengthvar1}-\ref{lengthvar3} (where $\delta f_{3}$ is arbitrary), we
have
\begin{align*}
\frac{\partial\gamma_{1,23}}{\partial f_{3}} &  =\frac{h_{13,2}}{\ell_{13}}\\
\frac{\partial\gamma_{2,13}}{\partial f_{3}} &  =\frac{h_{23,1}}{\ell_{23}}\\
\frac{\partial\gamma_{3,12}}{\partial f_{3}} &  =-\frac{h_{13,2}}{\ell_{13}%
}-\frac{h_{23,1}}{\ell_{23}}.
\end{align*}

\end{theorem}

\begin{proof}
We have already proven the first two equalities. The last follows from the
fact that in a Euclidean triangle, $\gamma_{1,23}+\gamma_{2,13}+\gamma
_{3,12}=\pi.$
\end{proof}

\begin{remark}
Special cases of this theorem appear in \cite{He}\cite{G1}\cite{G2}\cite{G3}
and less refined versions (where only signs and not explicit values are
computed for the derivatives) appear in many other places, including
\cite{Thurs}\cite{RS}\cite{MR}\cite{Ste}.
\end{remark}

\subsection{Three dimensions\label{section 3d angle vars}}

Now consider a tetrahedron $\left\{  1,2,3,4\right\}  $. Similarly, we will
need variations of the form
\begin{align}
\delta\ell_{12} &  =0,\label{3d length var 1}\\
\delta\ell_{13} &  =0,\label{3d length var 2}\\
\delta\ell_{23} &  =0,\label{3d length var 3}\\
\delta\ell_{14} &  =d_{41}\delta f_{4},\label{3d length var 4}\\
\delta\ell_{24} &  =d_{42}\delta f_{4},\label{3d length var 5}\\
\delta\ell_{34} &  =d_{43}\delta f_{4},\label{3d length var 6}%
\end{align}
where $d$ is a metric on $\left\{  1,2,3,4\right\}  $. For convenience, we
embed the vertices of the tetrahedron as $\left\{  v_{1},v_{2},v_{3}%
,v_{4}\right\}  $ so that it has the correct edge lengths and such that
$v_{1}$ is at the origin, $v_{2}$ is along the positive $x$-axis, $v_{3}$ is
in the $xy$-plane, and $v_{4}$ is above the $xy$-plane. We will need some
additional points. We let $v_{4}^{\prime}$ be the new vertex $4$ gotten by
taking lengths $\ell_{ij}+\delta\ell_{ij}$ (remembering that, for instance,
$\delta\ell_{12}=0$) and embedding this tetrahedron as $v_{1}v_{2}v_{3}%
v_{4}^{\prime},$ with $v_{4}^{\prime}$ above the $xy$-plane. We also need
$v_{4,12}^{\prime}$ which is the point on the plane $v_{1}v_{2}v_{4}^{\prime}$
which makes $v_{1}v_{2}v_{4,12}^{\prime}$ a triangle congruent to $v_{1}%
v_{2}v_{4}.$ Also, we have the point $v_{4,1}^{\prime},$ which is the point on
the line $v_{1}v_{4}^{\prime}$ which is a distance $\ell_{14}$ from $v_{1}.$
See Figure \ref{3dpic1}.%
\begin{figure}
[tb]
\begin{center}
\includegraphics[
natheight=3.212100in,
natwidth=2.688200in,
height=4.3661in,
width=3.7625in
]%
{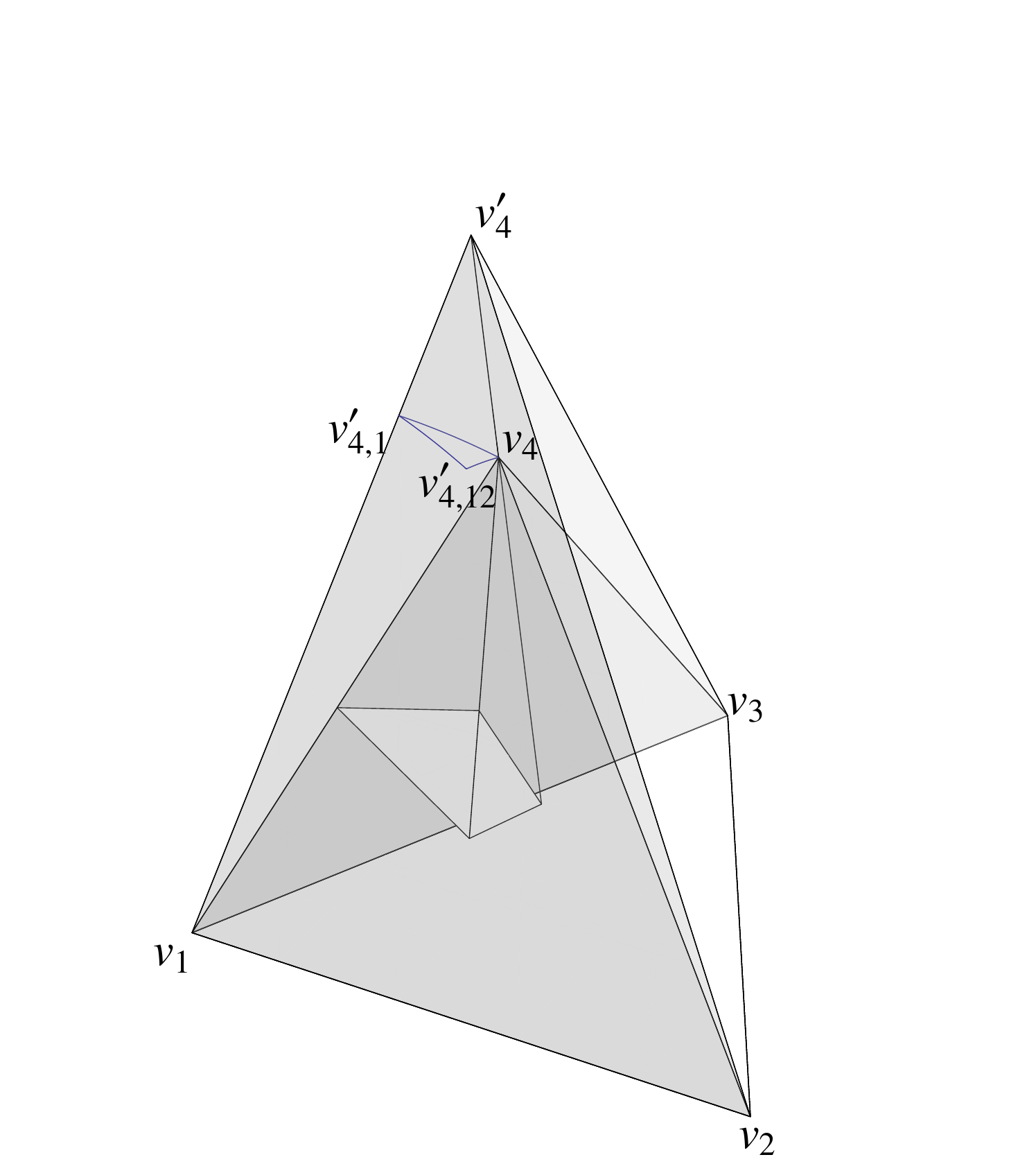}%
\caption{Variation of a tetrahedron.}%
\label{3dpic1}%
\end{center}
\end{figure}

We first observe that the right tetrahedron $v_{4}v_{4,12}^{\prime}%
v_{4,1}^{\prime}v_{4}^{\prime}$ is similar to the tetrahedron $c_{1234}%
c_{124}c_{14}v_{4}.$ This is because $c_{1234},$ $v_{4},$ and $v_{4}^{\prime}$
are colinear (the proof is exactly analogous to the proof of Proposition
\ref{straightline}, and is thus omitted). This implies that
\[
\frac{\left(  \delta\ell_{14}\right)  ^{2}}{d_{41}^{2}}=\frac{\frac{1}%
{2}\left(  \ell_{14}\delta\gamma_{1,24}\right)  \left(  \ell_{14}\sin
\gamma_{1,24}\left(  \delta\beta_{12,34}\right)  \right)  }{\frac{1}%
{2}h_{14,2}h_{124,3}}.
\]
We conclude that
\begin{align*}
\frac{h_{124,3}h_{14,2}}{d_{41}^{2}\sin\gamma_{1,24}\ell_{14}}  &
=\frac{\delta\gamma_{1,24}}{\delta\ell_{14}}\ell_{14}\left(  \frac{\delta
\beta_{12,34}}{\delta\ell_{14}}\right) \\
&  =\frac{h_{14,2}}{d_{41}}\left(  \frac{\delta\beta_{12,34}}{\delta\ell_{14}%
}\right)
\end{align*}
using Theorem \ref{2D}. We thus get%
\begin{equation}
\frac{h_{124,3}}{d_{41}\sin\gamma_{1,24}\ell_{14}}=\frac{\delta\beta_{12,34}%
}{\delta\ell_{14}}. \label{dihedral var fmla}%
\end{equation}

Furthermore, if $\alpha_{1}$ is the solid angle at vertex $v_{1},$ then we get
that $\ell_{14}^{2}\delta\alpha_{1}$ is approximately the sum of the areas of
two triangles on the sphere centered at $v_{1}$ of radius $\ell_{14}.$ One
triangle has vertices $v_{4,1}^{\prime},v_{4},$ and a point on the $x$-axis
which we will call $b.$ See Figure \ref{sphericaltripic}.
\begin{figure}
[tb]
\begin{center}
\includegraphics[
natheight=4.185900in,
natwidth=3.603100in,
height=4.3661in,
width=3.7625in
]%
{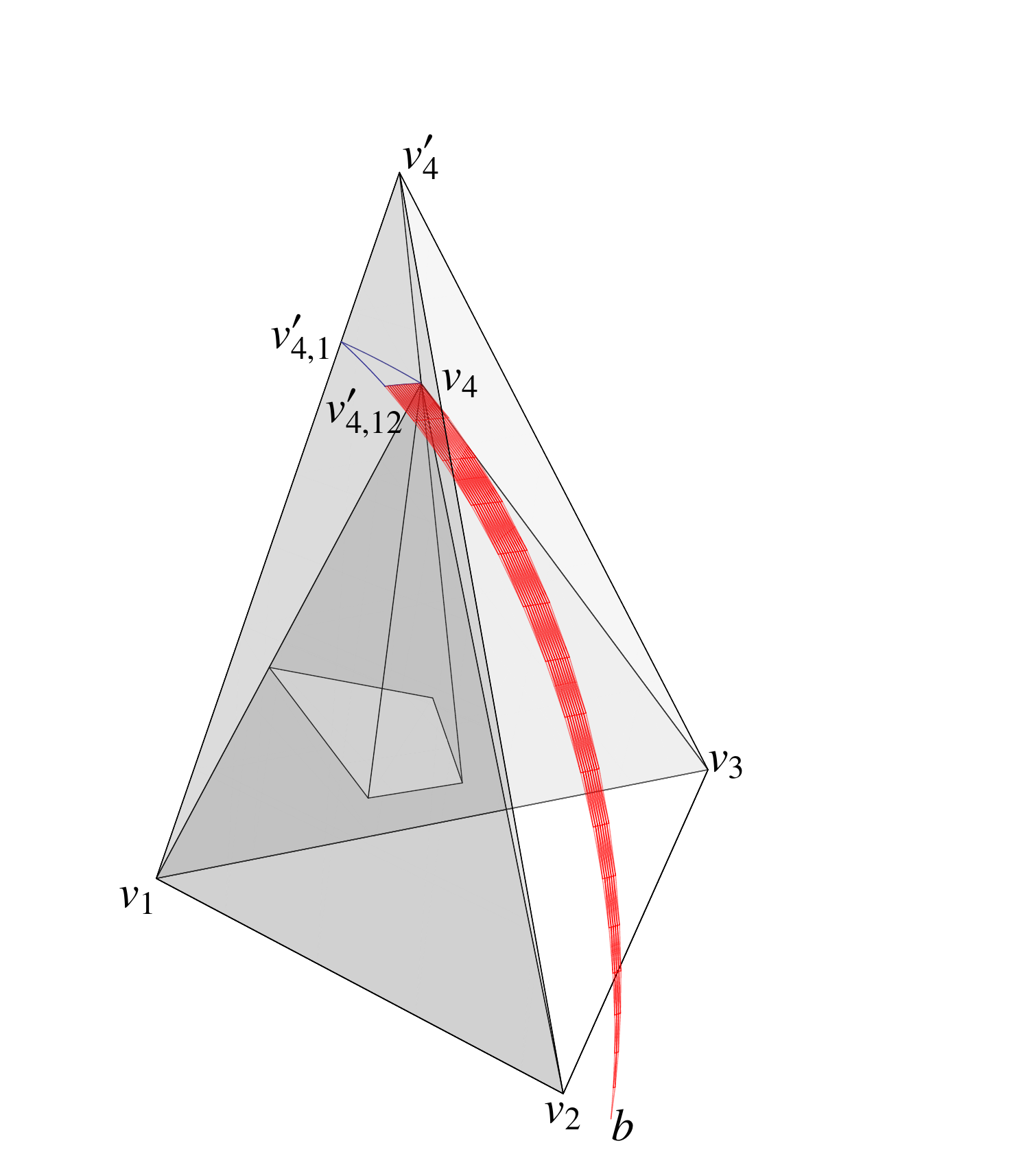}%
\caption{Angle variation setup for a tetrahedron.}%
\label{sphericaltripic}%
\end{center}
\end{figure}
This triangle can be divided into two spherical triangles, $v_{4,1}^{\prime
}v_{4,12}^{\prime}v_{4}$ and $v_{4,12}^{\prime}v_{4}b.$ Each of these
triangles has a right angle at $v_{4,12}^{\prime}.$ Note that the first
triangle has area which vanishes to higher order, so up to first order, the
area is the area of the second triangle. This triangle has a right angle at
$v_{4,12}^{\prime},$ has angle $\delta\beta_{12,34}$ at $b,$ and some other
angle at $v_{4}$, say $\frac{\pi}{2}-\gamma,$ where $\gamma$ is small. It is
easy to see that the side $v_{4,12}^{\prime}b$ in the spherical triangle has
length $\ell_{14}\delta\gamma_{1,24}$ and the side $v_{4,12}^{\prime}v_{4}$
has length $\ell_{14}\sin\gamma_{1,24}\delta\beta_{12,34}$ to first order (the
error is like $\left(  \delta\beta_{12,34}\right)  ^{3}$). We may compute
$\gamma$ since we know by spherical trigonometry that
\[
\cos\left(  \frac{\pi}{2}-\gamma\right)  =\frac{\tan\left(  \delta
\beta_{12,34}\sin\gamma_{1,24}\right)  }{\tan\gamma_{1,24}}.
\]
We look at this asymptotically where $\delta\beta_{12,34}$ is small, and see
that
\[
\gamma=\frac{\delta\beta_{12,34}\sin\gamma_{1,24}}{\tan\gamma_{1,24}}%
=\delta\beta_{12,34}\cos\gamma_{1,24}%
\]
and thus the area of the triangle is
\begin{align*}
&  \ell_{14}^{2}\left(  \delta\beta_{12,34}+\frac{\pi}{2}+\frac{\pi}{2}%
-\delta\beta_{12,34}\cos\gamma_{1,24}-\pi\right)  \\
&  =\ell_{14}^{2}\delta\beta_{12,34}\left(  1-\cos\gamma_{1,24}\right)  .
\end{align*}
Hence we have that
\[
\delta\alpha_{1}=\delta\beta_{12,34}+\delta\beta_{13,24}+\delta\beta_{14,23}%
\]
implies that
\[
\delta\beta_{12,34}\left(  1-\cos\gamma_{1,24}\right)  +\delta\beta
_{13,24}\left(  1-\cos\gamma_{1,34}\right)  =\delta\beta_{12,34}+\delta
\beta_{13,24}+\delta\beta_{14,23}%
\]
or%
\begin{align*}
\delta\beta_{14,23} &  =-\cos\gamma_{1,24}\delta\beta_{12,34}-\cos
\gamma_{1,34}\delta\beta_{13,24}\\
&  =-\cos\gamma_{1,24}\frac{h_{124,3}}{d_{41}\sin\gamma_{1,24}\ell_{14}}%
\delta\ell_{14}-\cos\gamma_{1,34}\frac{h_{134,2}}{d_{41}\sin\gamma_{1,34}%
\ell_{14}}\delta\ell_{14}\\
&  =-\frac{1}{d_{41}\ell_{14}}\left(  \cot\gamma_{1,24}h_{124,3}+\cot
\gamma_{1,34}h_{134,2}\right)  \delta\ell_{14}\\
&  =\frac{1}{d_{14}d_{41}\ell_{14}}\left(  h_{14,2}h_{124,3}+h_{14,3}%
h_{134,2}-\frac{d_{12}h_{124,3}}{\sin\gamma_{1,24}}-\frac{d_{13}h_{134,2}%
}{\sin\gamma_{1,34}}\right)  \delta\ell_{14}.
\end{align*}
Recalling that
\[
\frac{h_{124,3}}{d_{41}\sin\gamma_{1,24}\ell_{14}}=\frac{\delta\beta_{12,34}%
}{\delta\ell_{14}}%
\]
we get%
\begin{align}
&  d_{41}\left(  d_{14}\delta\beta_{14,23}+d_{12}\delta\beta_{12,34}%
+d_{13}\delta\beta_{13,24}\right)  \nonumber\\
&  =\frac{1}{\ell_{14}}\left(  h_{14,2}h_{124,3}+h_{14,3}h_{134,2}\right)
\delta\ell_{14}\nonumber\\
&  =2\frac{A_{14,23}}{\ell_{14}}\delta\ell_{14}=2\frac{A_{14,23}}{\ell_{14}%
}d_{41}\delta f_{4}.\label{dihedraldual1}%
\end{align}
That is,
\[
d_{14}\delta\beta_{14,23}+d_{12}\delta\beta_{12,34}+d_{13}\delta\beta
_{13,24}=2\frac{A_{14,23}}{\ell_{14}}\delta f_{4}.
\]
Furthermore, the Schl\"{a}fli formula implies that in the tetrahedron,
\[
\ell_{12}\delta\beta_{12}+\ell_{13}\delta\beta_{13}+\ell_{14}\delta\beta
_{14}+\ell_{23}\delta\beta_{23}+\ell_{24}\delta\beta_{24}+\ell_{34}\delta
\beta_{34}=0,
\]
and so%
\begin{align*}
0 &  =\left(  d_{14}\delta\beta_{14}+d_{12}\delta\beta_{12}+d_{13}\delta
\beta_{13}\right)  +\left(  d_{21}\delta\beta_{12}+d_{23}\delta\beta
_{23}+d_{24}\delta\beta_{24}\right)  \\
&  \;\;\;\;+\left(  d_{31}\delta\beta_{13}+d_{32}\delta\beta_{23}+d_{34}%
\delta\beta_{34}\right)  +\left(  d_{41}\delta\beta_{14}+d_{42}\delta
\beta_{24}+d_{43}\delta\beta_{34}\right)  .
\end{align*}
The first three terms are all of the form (\ref{dihedraldual1}) and the last
is different (because the lengths of edges around vertex $4$ are changing).
Thus,%
\begin{align*}
\left(  d_{41}\delta\beta_{14}+d_{42}\delta\beta_{24}+d_{43}\delta\beta
_{34}\right)   &  =-2\frac{A_{14,23}}{d_{41}\ell_{14}}\delta\ell_{14}%
-2\frac{A_{24,13}}{d_{42}\ell_{24}}\delta\ell_{24}-2\frac{A_{34,12}}%
{d_{43}\ell_{34}}\delta\ell_{34}\\
&  =\left[  -2\frac{A_{14,23}}{\ell_{14}}-2\frac{A_{24,13}}{\ell_{24}}%
-2\frac{A_{34,12}}{\ell_{34}}\right]  \delta f_{4}.
\end{align*}
We have just proven the following theorem.

\begin{theorem}
\label{3D}Under variations of the form (\ref{3d length var 1}%
)-(\ref{3d length var 6}), we have
\begin{align*}
d_{14}\delta\beta_{14}+d_{12}\delta\beta_{12}+d_{13}\delta\beta_{13}  &
=2\frac{A_{14,23}}{\ell_{14}}\delta f_{4},\\
d_{24}\delta\beta_{24}+d_{23}\delta\beta_{23}+d_{21}\delta\beta_{12}  &
=2\frac{A_{24,13}}{\ell_{24}}\delta f_{4},\\
d_{34}\delta\beta_{34}+d_{32}\delta\beta_{23}+d_{31}\delta\beta_{13}  &
=2\frac{A_{34,12}}{\ell_{34}}\delta f_{4},\\
d_{41}\delta\beta_{14}+d_{42}\delta\beta_{24}+d_{43}\delta\beta_{34}  &
=\left[  -2\frac{A_{14,23}}{\ell_{14}}-2\frac{A_{24,13}}{\ell_{24}}%
-2\frac{A_{34,12}}{\ell_{34}}\right]  \delta f_{4}.
\end{align*}

\end{theorem}

We actually derived a finer result, with explicit computation of the
variations of individual dihedral angles. However, the result in this form is
more compactly stated and all we will use in the remainder of this paper.
Also, this result was derived in \cite{G1} for the specific case of
sphere-packing configurations of a tetrahedron. Less precise results along
these lines were examined in the sphere-packing case in \cite{CR}\cite{Riv2}
as well. We will go into more detail later.

\section{Curvature variations\label{section curvature variations}}

\subsection{Two dimensions}

Discrete curvature in two dimensions has been well studied, with curvature as
in Definition \ref{definition 2d curvature}. Theorem \ref{2D} has the
following implication.

\begin{theorem}
\label{curv derivative 2D}Let $f\left(  t\right)  $ be a conformal variation
of $\left(  M^{2},T,d\right)  $ Then for $t=0,$
\begin{equation}
\frac{dK_{i}}{dt}=-\sum_{j}\frac{\ell_{ij}^{\ast}}{\ell_{ij}}\left(
\frac{df_{j}}{dt}-\frac{df_{i}}{dt}\right)  .\label{2d curvature time deriv}%
\end{equation}

\end{theorem}

\begin{proof}
We compute
\begin{align*}
\frac{dK_{i}}{dt} &  =\frac{d}{dt}\left(  2\pi-\sum_{j,k}\gamma_{i,jk}\right)
\\
&  =-\sum_{j,k}\left(  \frac{\partial\gamma_{i,jk}}{\partial f_{i}}%
\frac{df_{i}}{dt}+\frac{\partial\gamma_{i,jk}}{\partial f_{j}}\frac{df_{j}%
}{dt}+\frac{\partial\gamma_{i,jk}}{\partial f_{k}}\frac{df_{k}}{dt}\right)  \\
&  =-\sum_{j,k}\left[  \frac{h_{ij,k}}{\ell_{ij}}\left(  \frac{df_{j}}%
{dt}-\frac{df_{i}}{dt}\right)  +\frac{h_{ik,j}}{\ell_{ik}}\left(  \frac
{df_{k}}{dt}-\frac{df_{i}}{dt}\right)  \right]  ,
\end{align*}
which implies the result.
\end{proof}

\begin{remark}
We will express many of the variation results as in Theorem
\ref{curv derivative 2D} instead of in terms of $\frac{\partial K_{i}%
}{\partial f_{j}}$ in order to emphasize the presence of the Laplacian. Notice
that the formula (\ref{2d curvature time deriv}) has the form
\[
\frac{dK_{i}}{dt}=-\left(  \triangle\frac{df}{dt}\right)  _{i}%
\]
for an appropriate definition of the Laplacian $\triangle$. We will comment
more on this in Section \ref{section laplace}.
\end{remark}

The formulas from Theorem \ref{2D} also imply that the curvatures are
variational giving the proof of Theorem \ref{2d variational}.

\begin{proof}
[Proof of Theorem \ref{2d variational}]$F$ will be defined as
\[
F\left(  f\right)  =\int_{f_{0}}^{f}\omega
\]
for the $1$-form
\[
\omega=\sum K_{i}df_{i}.
\]
To ensure that the integral is independent of path, we need that $\omega$ is
closed. Since we are in a simply connected domain, we need only check that
\[
\frac{\partial K_{i}}{\partial f_{j}}=\frac{\partial K_{j}}{\partial f_{i}}%
\]
for $i\neq j.$ We can compute these derivative explicitly, and they are
\[
\frac{\partial K_{i}}{\partial f_{j}}=-\frac{\ell_{ij}^{\ast}}{\ell_{ij}%
}=\frac{\partial K_{j}}{\partial f_{i}}%
\]
if $i$ and $j$ share an edge and zero otherwise.
\end{proof}

This formulation of the prescribed curvature problem in a variational
framework has been studied by many people. See, for instance, \cite{DGL}
\cite{Riv1} \cite{CdV} \cite{CL} \cite{Spr} \cite{Guo}, many of which derive
the functional in precisely the same way. Certain conformal variations in the
discrete setting have been proposed by Ro\v{c}ek and Williams in the context
of Regge calculus \cite{RW}, Thurston in the setting of circle patterns
\cite{Thurs} (see also other work on circle packing, e.g., \cite{Ste}), and
Luo \cite{Luo1} (see also \cite{SSP}). Our current setting takes each of these
definitions and proofs as special cases (see Section
\ref{section special cases}). In many of these papers it is shown that the
largest \textquotedblleft reasonable\textquotedblright\ domain for the $f$'s
is simply connected. The advantage of Theorem \ref{2d variational} is that it
works for a more general class of conformal variations. In addition, we have a
geometric description of the derivatives $\frac{\partial K_{i}}{\partial
f_{j}}$ which is absent from most of these previous works.

\subsection{Three dimensions and the Einstein-Hilbert-Regge
functional\label{section 3d curv vars}}

We could follow a similar method to that in Theorem \ref{2d variational} to
prove that there are three-dimensional curvatures which are variational.
However, we will present this fact in a different way by using the
Einstein-Hilbert-Regge functional from Definition \ref{definition EHR}.

Recall Definition \ref{definition 3d scalar curvature} for scalar curvatures
in dimension three. The definition is first motivated by seeing how the
Schl\"{a}fli formula decomposes as sums around vertices. The Schl\"{a}fli
formula on a tetrahedron is
\[
\sum_{i,j}\ell_{ij}\delta\beta_{ij}=0,
\]
where the sum is over all edges $\left\{  i,j\right\}  $ in the tetrahedron.
It can be written as
\begin{align*}
0 &  =\left(  d_{12}+d_{21}\right)  \delta\beta_{12}+\left(  d_{13}%
+d_{31}\right)  \delta\beta_{13}+\left(  d_{14}+d_{41}\right)  \delta
\beta_{14}\\
&  \;\;\;+\left(  d_{23}+d_{32}\right)  \delta\beta_{23}+\left(  d_{24}%
+d_{42}\right)  \delta\beta_{24}+\left(  d_{34}+d_{43}\right)  \delta
\beta_{34}\\
&  =\left(  d_{12}\delta\beta_{12}+d_{13}\delta\beta_{13}+d_{14}\delta
\beta_{14}\right)  +\left(  d_{21}\delta\beta_{12}+d_{23}\delta\beta
_{23}+d_{24}\delta\beta_{24}\right)  \\
&  \;\;\;+\left(  d_{31}\delta\beta_{13}+d_{32}\delta\beta_{23}+d_{34}%
\delta\beta_{34}\right)  +\left(  d_{41}\delta\beta_{14}+d_{42}\delta
\beta_{24}+d_{43}\delta\beta_{34}\right)  ,
\end{align*}
giving the vertex breakdown motivating the curvature formula. The Schl\"{a}fli
formula allows an easy computation of first derivatives of the
Einstein-Hilbert-Regge functional on a triangulation of a closed manifold,
giving%
\begin{align}
\frac{\partial}{\partial f_{i}}\mathcal{EHR}\left(  T,\ell\left(  f\right)
\right)   &  =\frac{\partial}{\partial f_{i}}\left[  \sum_{i,j}\left(
2\pi-\sum_{k,\ell}\beta_{ij,k\ell}\right)  \ell_{ij}\right]  \nonumber\\
&  =\sum_{j}\left(  2\pi-\sum_{k,\ell}\beta_{ij,k\ell}\right)  d_{ij}%
\label{conformal variation of EHR}\\
&  =K_{i}.\nonumber
\end{align}

To compute the second variation, we need the variation of $K_{i}.$ Using
Theorem \ref{3D} we can express the derivatives of curvature.

\begin{theorem}
\label{theorem 3d curvature variation}Let $f\left(  t\right)  $ be a conformal
variation of $\left(  M^{3},T,d\right)  .$ Then,
\begin{align}
\frac{dK_{i}}{dt} &  =-2\sum_{j\neq i}\frac{\ell_{ij}^{\ast}}{\ell_{ij}%
}\left(  \frac{df_{j}}{dt}-\frac{df_{i}}{dt}\right)  +\sum_{j\neq i}\left(
2\pi-\sum_{k,\ell}\beta_{ij,k\ell}\right)  \frac{d}{dt}d_{ij}%
\label{curvature variation fmla 1}\\
&  =-\sum_{j\neq i}\left(  2\frac{\ell_{ij}^{\ast}}{\ell_{ij}}-\frac{q_{ij}%
}{\ell_{ij}}K_{ij}\right)  \left(  \frac{df_{j}}{dt}-\frac{df_{i}}{dt}\right)
+K_{i}\frac{df_{i}}{dt},\nonumber
\end{align}
where
\[
q_{ij}=\frac{\partial d_{ij}}{\partial f_{j}}=\frac{\partial d_{ji}}{\partial
f_{i}}.
\]

\end{theorem}

We note that the first term in (\ref{curvature variation fmla 1}) is the
Laplacian operator studied by the author in \cite{G3}.

\begin{proof}
We compute%
\begin{align*}
\frac{d}{dt}K_{i} &  =2\pi\sum_{j}\frac{d}{dt}d_{ij}-2\sum_{j\neq i}\frac
{\ell_{ij}^{\ast}}{\ell_{ij}}\left(  \frac{df_{j}}{dt}-\frac{df_{i}}%
{dt}\right)  \\
&  \;\;\;-\sum_{j,k,\ell}\left(  \beta_{ij,k\ell}\frac{d}{dt}d_{ij}%
+\beta_{ik,j\ell}\frac{d}{dt}d_{ik}+\beta_{i\ell,jk}\frac{d}{dt}d_{i\ell
}\right)  \\
&  =-2\sum_{j\neq i}\frac{\ell_{ij}^{\ast}}{\ell_{ij}}\left(  \frac{df_{j}%
}{dt}-\frac{df_{i}}{dt}\right)  +\sum_{j\neq i}\left(  2\pi-\sum_{k,\ell}%
\beta_{ij,k\ell}\right)  \frac{d}{dt}d_{ij}.
\end{align*}
Furthermore, if there is a conformal structure, then
\[
\frac{\partial d_{ij}}{\partial f_{j}}=\frac{\partial}{\partial f_{j}}%
\frac{\partial\ell_{ij}}{\partial f_{i}}=\frac{\partial}{\partial f_{i}}%
\frac{\partial\ell_{ij}}{\partial f_{j}}=\frac{\partial d_{ji}}{\partial
f_{i}}.
\]
Thus
\[
q_{ij}=\frac{\partial d_{ij}}{\partial f_{j}}=\frac{\partial d_{ji}}{\partial
f_{i}}=\frac{\partial^{2}\ell_{ij}}{\partial f_{i}\partial f_{j}}%
\]
is symmetric, i.e., $q_{ij}=q_{ji}.$ This also implies that
\begin{align*}
\frac{\partial d_{ij}}{\partial f_{i}} &  =\frac{\partial\ell_{ij}}{\partial
f_{i}}-\frac{\partial d_{ji}}{\partial f_{i}}\\
&  =d_{ij}-q_{ij}.
\end{align*}
Thus
\[
\frac{d}{dt}d_{ij}=\left(  d_{ij}-q_{ij}\right)  \frac{df_{i}}{dt}+q_{ij}%
\frac{df_{j}}{dt}.
\]
The result follows.
\end{proof}

We can compute the variations of the Einstein-Hilbert-Regge functional.

\begin{proof}
[Proof of Theorem \ref{theorem conformal 3d}]Using
(\ref{conformal variation of EHR}), we see immediately that
\[
\frac{d}{dt}\mathcal{EHR}\left(  M,T,\ell\left(  f\left(  t\right)  \right)
\right)  =\sum_{i\in V}K_{i}\frac{df_{i}}{dt}.
\]
Using Theorem \ref{theorem 3d curvature variation}, we compute that
\begin{align*}
\frac{d^{2}}{dt^{2}}\mathcal{EHR}\left(  M,T,\ell\left(  f\left(  t\right)
\right)  \right)   &  =\sum_{i}\frac{dK_{i}}{dt}\frac{df_{i}}{dt}+K_{i}%
\frac{d^{2}f_{i}}{dt^{2}}\\
&  =-\sum_{i}\sum_{j\neq i}\left(  2\frac{\ell_{ij}^{\ast}}{\ell_{ij}}%
-\frac{q_{ij}}{\ell_{ij}}K_{ij}\right)  \left(  \frac{df_{j}}{dt}-\frac
{df_{i}}{dt}\right)  \frac{df_{i}}{dt}\\
&  \;\;\;\;\;\;\;+\sum_{i}K_{i}\left[  \left(  \frac{df_{i}}{dt}\right)
^{2}+\frac{d^{2}f_{i}}{dt^{2}}\right]  \\
&  =\sum_{i}\sum_{j\neq i}\left(  \frac{\ell_{ij}^{\ast}}{\ell_{ij}}%
-\frac{q_{ij}}{2\ell_{ij}}K_{ij}\right)  \left(  \frac{df_{j}}{dt}%
-\frac{df_{i}}{dt}\right)  ^{2}\\
&  \;\;\;\;\;\;\;+\sum_{i}K_{i}\left[  \left(  \frac{df_{i}}{dt}\right)
^{2}+\frac{d^{2}f_{i}}{dt^{2}}\right]  .
\end{align*}
The rest of the theorem follows immediately from
(\ref{conformal variation of EHR}) and Regge's variation formula
(\ref{regge variation}).
\end{proof}

In order to get better control of $q_{ij}$, we will look at special conformal
structures in Section \ref{section special cases}.

Finally we can complete the proof of Corollary
\ref{corollary volume variation}.

\begin{proof}
[Proof of Corollary \ref{corollary volume variation}]We see from Figure
\ref{3dpic1} that we must have that if we fix $f_{1},f_{2},f_{3}$ and let
$f_{4}$ vary, then
\begin{align*}
\delta V_{1234} &  =\frac{1}{3}A_{124}\ell_{14}\sin\gamma_{1,24}\delta
\beta_{12,34}+\frac{1}{3}A_{134}\ell_{14}\sin\gamma_{1,34}\delta\beta
_{13,24}\\
&  \;\;\;\;\;\;+\frac{1}{3}A_{234}\ell_{34}\sin\gamma_{3,24}\delta
\beta_{23,14}\\
&  =\frac{1}{3}\left(  A_{124}h_{124,3}+A_{134}h_{134,2}+A_{234}%
h_{234,1}\right)
\end{align*}
by (\ref{dihedral var fmla}), which implies that
\[
\frac{\partial}{\partial f_{i}}\mathcal{V}\left(  T,\ell\left(  f\right)
\right)  =V_{i}%
\]
where
\[
V_{i}=\frac{1}{3}\sum_{j,k,\ell}h_{ijk,\ell}A_{ijk}%
\]
and the sum is over all tetrahedra containing $i$ and all faces in those
tetrahedra containing $i.$
\end{proof}

\section{Examples of conformal structures\label{section special cases}}

In this section we place previously studied geometric structures into the
framework of conformal structures.

\subsection{Circle and sphere packing\label{section sphere packing}}

The case of circle packing and sphere packing is when edge lengths arise from
spheres centered at the vertices which are externally tangent to each other.
In this case, there are positive weights $r_{i}$ corresponding to the radii
and $\ell_{ij}=r_{i}+r_{j}$. In two dimensions, circle packings have been
considered in a number of contexts; see Stephenson's monograph \cite{Ste} for
an overview. In three dimensions, this case was considered by Cooper-Rivin
\cite{CR}. They noticed, in particular, that for a sphere packing, one can
rewrite the Schl\"{a}fli formula in the following way:%

\begin{align*}
0  &  =\left(  r_{1}+r_{2}\right)  \delta\beta_{12}+\left(  r_{1}%
+r_{3}\right)  \delta\beta_{13}+\left(  r_{1}+r_{4}\right)  \delta\beta_{14}\\
&  \;\;\;+\left(  r_{2}+r_{3}\right)  \delta\beta_{23}+\left(  r_{2}%
+r_{4}\right)  \delta\beta_{24}+\left(  r_{3}+r_{4}\right)  \delta\beta_{34}\\
&  =r_{1}\left(  \delta\beta_{12}+\delta\beta_{13}+\delta\beta_{14}\right)
+r_{2}\left(  \delta\beta_{12}+\delta\beta_{23}+\delta\beta_{24}\right) \\
&  \;\;\;+r_{3}\left(  \delta\beta_{13}+\delta\beta_{23}+\delta\beta
_{34}\right)  +r_{4}\left(  \delta\beta_{14}+\delta\beta_{24}+\delta\beta
_{34}\right) \\
&  =r_{1}\delta\alpha_{1}+r_{2}\delta\alpha_{2}+r_{3}\delta\alpha_{3}%
+r_{4}\delta\alpha_{4},
\end{align*}
where $\alpha_{i}$ is the solid angle at vertex $i,$ and thus $\delta
\alpha_{1}=\delta\left(  \beta_{12}+\beta_{13}+\beta_{14}-\pi\right)  =$
$\delta\beta_{12}+\delta\beta_{13}+\delta\beta_{14}.$ They used this to
motivate the definition of scalar curvature as $4\pi-\sum\alpha_{i}$ where the
sum is over all tetrahedra containing $i$ as a vertex. From our setting, we
would define the scalar curvature measure instead as%

\[
K_{i}=\left(  4\pi-\sum_{j,k,\ell}\alpha_{i,jk\ell}\right)  r_{i}=\sum
_{j}\left(  2\pi-\sum_{k,\ell}\beta_{ij,k\ell}\right)  r_{i},
\]
where in the right side, the first sum is over all edges incident on $i$ and
the second sum is the sum over all tetrahedra containing $\left\{
i,j\right\}  $ as an edge. The second equality can be easily derived using the
Euler characteristic and area formula of the sphere centered at vertex $i.$

To match this to our setting, we see that we must take $f_{i}=\log r_{i}$ and
$d_{ij}=r_{i},$ since
\begin{align*}
\ell_{ij}  &  =r_{i}+r_{j}\\
\frac{\partial\ell_{ij}}{\partial r_{i}}r_{i}  &  =r_{i}.
\end{align*}
Thus we have the following conformal structure.

\begin{definition}
The \emph{circle/sphere packing conformal structure,} $\mathcal{C}^{P}\left(
M,T\right)  ,$ is the map defining
\[
d_{ij}=e^{f_{i}}%
\]
for every oriented edge in $E_{+}\left(  T\right)  $ restricted to an
appropriate domain of $f\in V^{\ast}\left(  T\right)  .$
\end{definition}

In two dimensions, the triangle inequality is automatically satisfied, and so
the domain is all of $V^{\ast}\left(  T\right)  .$ In three dimensions, there
is an additional condition that the square volumes of three-dimensional
simplices (as defined by the Cayley-Menger determinant formula) are positive.
This is discussed in some detail in \cite{G1} \cite{G2}.

We see that the formulas in Section \ref{section 3d angle vars} correspond to
\[
\frac{\partial\alpha_{1}}{\partial r_{4}}r_{1}r_{4}=2\frac{A_{14,23}}%
{\ell_{14}}.
\]
This is the same formula derived by the author in \cite{G1}.

\subsection{Fixed intersection angles/inversive distance}

There is a more general case of circles or spheres with fixed intersection
angles, originally considered by Thurston \cite{Thurs}. Here we parametrize
lengths by two parameters, radii $r_{i}$ and inversive distances $\eta_{ij}$.
The inversive distance (see, for instance, \cite{Guo}) is like the cosine of
the supplement of the intersection angle, defined so that
\[
\ell_{ij}^{2}=r_{i}^{2}+r_{j}^{2}+2r_{i}r_{j}\eta_{ij}.
\]
We will use this formula to parametrize the lengths by the radii $r_{i}$ with
inversive distances fixed. It essentially corresponds to having circles at the
vertices of radius $r_{i}$ and intersecting at angle $\arccos\left(
-\eta_{ij}\right)  .$ If $\eta_{ij}$ is not between $-1$ and $1,$ then the
circles may not intersect, but this is not a problem for the theory. There is
always a circle orthogonal to these circles, and we take the center of the
triangle to be the center of this orthocircle. (Note, it is possible that this
circle does not have real radius, but the center is still well defined using
the algebra of circles given in \cite{Ped}.) We then find that
\begin{equation}
d_{ij}=\frac{r_{i}\left(  r_{i}+r_{j}\eta_{ij}\right)  }{\ell_{ij}%
}.\label{dij for fixed inversive distance}%
\end{equation}
For a path in the $r$ variables, we compute
\begin{align*}
\ell_{ij}\frac{d}{dt}\ell_{ij} &  =\left(  r_{i}+r_{j}\eta_{ij}\right)
\frac{dr_{i}}{dt}+\left(  r_{j}+r_{i}\eta_{ij}\right)  \frac{dr_{j}}{dt}\\
\frac{d}{dt}\ell_{ij} &  =d_{ij}\frac{1}{r_{i}}\frac{dr_{i}}{dt}+d_{ji}%
\frac{1}{r_{j}}\frac{dr_{j}}{dt}.
\end{align*}
Thus we see that $f_{i}=\log r_{i},$ giving the fixed inversive distance
conformal class.

\begin{definition}
For a given $\eta\in E\left(  T\right)  ^{\ast},$ the \emph{fixed inversive
distance conformal structure}, $\mathcal{C}^{FI}\left(  M,T,\eta\right)  ,$ is
the conformal structure described by the map
\[
d_{ij}=\frac{e^{f_{i}}\left(  e^{f_{i}}+e^{f_{j}}\eta_{ij}\right)  }{\ell
_{ij}\left(  f\right)  },
\]
where
\[
\ell_{ij}\left(  f\right)  =\sqrt{e^{2f_{i}}+e^{2f_{j}}+2e^{f_{i}}e^{f_{j}%
}\eta_{ij}}%
\]
is the length, when restricted to a proper domain in $V\left(  T\right)
^{\ast}$
\end{definition}

Note that there are some restrictions on the domain which may be quite
complicated, including the triangle inequality. However, it has been found
that in two dimensions, if $\eta_{ij}\geq0$ for all $\eta\in E\left(
T\right)  ^{\ast},$ then the domain is simply connected. This was initially
shown for $0\leq\eta_{ij}\leq1$ by Thurston (\cite{Thurs} \cite{MR}) and the
additional cases were proven recently by Guo \cite{Guo}.

We see that%
\begin{align*}
q_{ij}  &  =\frac{\partial d_{ij}}{\partial f_{j}}=r_{j}\frac{\partial d_{ij}%
}{\partial r_{j}}\\
&  =\frac{r_{i}^{2}r_{j}^{2}\left(  \eta_{ij}^{2}-1\right)  }{\ell_{ij}^{3}}.
\end{align*}
We finally get
\[
\frac{dK_{i}}{dt}=-\sum_{j\neq i}\left(  2\frac{\ell_{ij}^{\ast}}{\ell_{ij}%
}-\frac{r_{i}^{2}r_{j}^{2}\left(  1-\eta_{ij}^{2}\right)  }{\ell_{ij}^{4}%
}K_{ij}\right)  \left(  \frac{df_{j}}{dt}-\frac{df_{i}}{dt}\right)
+K_{i}\frac{df_{i}}{dt},
\]
and the second variation of the Einstein-Hilbert-Regge functional is%

\begin{align*}
\frac{d^{2}}{dt^{2}}\mathcal{EHR}\left(  T,\ell\left(  f\left(  t\right)
\right)  \right)   &  =\sum_{i}\sum_{j\neq i}\left(  \frac{\ell_{ij}^{\ast}%
}{\ell_{ij}}-\frac{r_{i}^{2}r_{j}^{2}\left(  1-\eta_{ij}^{2}\right)  }%
{2\ell_{ij}^{4}}K_{ij}\right)  \left(  \frac{df_{j}}{dt}-\frac{df_{i}}%
{dt}\right)  ^{2}\\
&  \;\;\;+\sum_{i}K_{i}\left[  \left(  \frac{df_{i}}{dt}\right)  ^{2}%
+\frac{d^{2}f_{i}}{dt^{2}}\right]  .
\end{align*}
Note that in the case that $\eta_{ij}=1,$ corresponding to sphere packing, the
second term is zero. In general, for spheres with intersection we have
$\eta_{ij}\leq1$ and for spheres which do not intersect we have $\eta_{ij}>1$
and so in each case the term with edge curvatures has a particular sign.

\subsection{Perpendicular bisectors\label{section perpendicular bisector}}

Here we give the conformal structure proposed by Ro\v{c}ek-Williams \cite{RW},
Luo \cite{Luo1}, and Pinkall-Schroeder-Springborn \cite{SSP}. This structure
has also been found in the numerical analysis literature on approximations of
the Laplacian in the context of the box method (see, e.g., \cite{Kerk} and
\cite{PC}). Take
\[
\ell_{ij}=e^{u_{i}+u_{j}}L_{ij}%
\]
where $L_{ij}$ are fixed lengths. We see that, given a path in the space of
$u$ variables,
\[
\frac{d}{dt}\ell_{ij}=\ell_{ij}\left(  \frac{du_{i}}{dt}+\frac{du_{j}}%
{dt}\right)  .
\]
If we take
\begin{align*}
d_{ij} &  =\frac{\ell_{ij}}{2}\\
f_{i} &  =2u_{i}%
\end{align*}
then
\[
\frac{d}{dt}\ell_{ij}=d_{ij}\frac{df_{i}}{dt}+d_{ji}\frac{df_{j}}{dt}.
\]
We notice that the duals to the edges intersect the edges at their midpoints,
which is why we call this the perpendicular bisector conformal structure
following \cite{Kerk}. It can be proven inductively that the center of any
simplex is the center of the sphere circumscribing that simplex.

\begin{definition}
Let $L\in E\left(  T\right)  $ be such that $\left(  M,T,L\right)  $ is a
piecewise flat manifold. The \emph{perpendicular bisector conformal
structure}, $\mathcal{C}^{PB}\left(  M,T,L\right)  ,$ is the conformal
structure determined by
\[
d_{ij}=\frac{1}{2}e^{\frac{1}{2}\left(  f_{i}+f_{j}\right)  }L_{ij},
\]
when restricted to an appropriate domain.
\end{definition}

Since $\left(  M,T,d\left(  \vec{0}\right)  \right)  $ is a piecewise flat,
metrized manifold, this conformal structure exists for $f_{i}$ close to $0.$
However, the largest possible domain must satisfy a number of inequalities.

We see that
\[
q_{ij}=\frac{\partial d_{ij}}{\partial f_{j}}=\frac{1}{4}\ell_{ij}.
\]
Thus, in three dimensions the variation of curvature is
\[
\frac{dK_{i}}{dt}=-\sum_{j\neq i}\left(  2\frac{\ell_{ij}^{\ast}}{\ell_{ij}%
}-\frac{1}{4}K_{ij}\right)  \left(  \frac{df_{j}}{dt}-\frac{df_{i}}%
{dt}\right)  +K_{i}\frac{df_{i}}{dt}.
\]
We get that
\begin{align*}
\frac{d^{2}}{dt^{2}}\mathcal{EHR}\left(  T,\ell\left(  f\left(  t\right)
\right)  \right)   &  =\sum_{i}\sum_{j\neq i}\left(  \frac{\ell_{ij}^{\ast}%
}{\ell_{ij}}-\frac{1}{8}K_{ij}\right)  \left(  \frac{df_{j}}{dt}-\frac{df_{i}%
}{dt}\right)  ^{2}\\
&  \;\;\;\;\;\;\;+\sum_{i}K_{i}\left[  \left(  \frac{df_{i}}{dt}\right)
^{2}+\frac{d^{2}f_{i}}{dt^{2}}\right]  .
\end{align*}

\section{The discrete Laplacian and the second
variation\label{section laplace}}

\subsection{Laplacians}

The relationship between the second variation of the functionals presented
here and the Laplacian is the main reason we describe these variations as
conformal. The standard Laplacian is defined as follows.

\begin{definition}
Let $\left(  M,T,d\right)  $ be a piecewise flat, metrized manifold. The
\emph{discrete Laplacian} $\triangle$ is an operator $V\left(  T\right)
^{\ast}\rightarrow V\left(  T\right)  ^{\ast}$ defined by
\[
\left(  \triangle\phi\right)  _{i}=\sum_{j}\frac{\ell_{ij}^{\ast}}{\ell_{ij}%
}\left(  \phi_{j}-\phi_{i}\right)
\]
for each vertex $i,$ where $\ell_{ij}^{\ast}$ is the dual length defined
appropriately (see Definitions \ref{definition 2d dual length} and
\ref{definition 3d dual length} and \cite{G3} for the general case).
\end{definition}

These can be considered Laplacians on the graph of the 1-skeleton with edges
weighted by $\frac{\ell_{ij}^{\ast}}{\ell_{ij}}.$ (For more on Laplacians on
graphs, see \cite{Chun}.) This is a very natural choice of Laplacian, arising,
for instance, by considering another function $\psi$ on the vertices, and
defining the Laplacian weakly as
\[
\sum_{i}\triangle\phi_{i}\psi_{i}=-\frac{n}{2}\sum_{i,j}\frac{\phi_{i}%
-\phi_{j}}{\ell_{ij}}\frac{\psi_{i}-\psi_{j}}{\ell_{ij}}V_{ij},
\]
for all choices of $\psi,$ where $V_{ij}$ is the volume associated to an edge,
defined by
\[
V_{ij}=\frac{1}{n}\ell_{ij}^{\ast}\ell_{ij},
\]
where $n$ is the dimension. This is an analogue of the definition of the
smooth Laplacian on a closed manifold as the operator such that
\[
\int\triangle\phi~\psi~dV=-\int\nabla\phi\cdot\nabla\psi~dV
\]
for all smooth functions $\psi.$ Another interesting observation about the
Laplacian is that the weights $\frac{\ell_{ij}^{\ast}}{\ell_{ij}}$ are very
much like conductances, in that they are inversely proportional to length and
directly proportional to cross-sectional area if one considers current through
wires located at the edges of the triangulation.

Laplacians of this geometric form have been studied for some time. The most
well-known is the \textquotedblleft cotan formula\textquotedblright\ for a
Laplacian on a planar triangulation. If one considers the perpendicular
bisector formulation of Section \ref{section perpendicular bisector} on a
planar domain or surface, one finds that $\ell_{ij}^{\ast}=\ell_{ij}\left(
\cot\gamma_{k,ij}+\cot\gamma_{\ell,ij}\right)  .$ It turns out that this is
precisely the finite element approximation of the Laplacian, as first computed
by Duffin \cite{Duf}. The cotan formula has been well-studied both in regards
to approximation of the Laplacian on domains and approximation of the
Laplacian on surfaces for computing minimal surfaces and bending energies.
See, e.g., \cite{PP} \cite{Kerk} \cite{HPW} \cite{WBHZG} \cite{BS}. In
addition, Laplacians have appeared in the study of circle packings. In fact,
to our knowledge, the first observation that variations of angles are related
to dual lengths dates to Z. He \cite{He} in the circle packing setting, where
it was used for constructing a Laplacian. Further work in two dimensions in
the setting of circle packings and circle diagrams with fixed inversive
distance which connects angle variations with Laplacians can be found in
\cite{Dub} \cite{CL} \cite{G4} \cite{Guo}. An interesting study of possible
Laplacians from a axiomatic development can be found in \cite{WMKG}.

\subsection{Properties of the Laplacian}

There are two properties of the smooth Laplacian which are desirable to have
in a discrete Laplacian:

\begin{enumerate}
\item \label{property 1 laplace}$\triangle$ is a negative semidefinite
operator with zero eigenspace corresponding exactly to constant functions
($\phi$ is a constant function if there exists $c\in\mathbb{R}$ such that
$\phi_{i}=c$ for all $i\in V$).

\item \label{property 2 laplace}$\triangle$ satisfies the weak maximum
principle, i.e., for any $\phi\in V\left(  T\right)  ^{\ast},$ if $\phi
_{m}=\min_{i}\phi_{i}$ and $\phi_{M}=\max_{i}\phi_{i}$ then $\triangle\phi
_{m}\geq0$ and $\triangle\phi_{M}\leq0.$
\end{enumerate}

Note that the definition of the Laplacian ensures that the constant functions
are in the nullspace. The second property is implied by $\ell_{ij}^{\ast}%
\geq0$ for all edges $\left\{  i,j\right\}  .$ Furthermore, we shall show that
the strict inequality $\ell_{ij}^{\ast}>0$ implies the property
\ref{property 1 laplace}. The Laplacian is a symmetric operator, and so it has
a full set of eigenvalues. If $\lambda$ is an eigenvalue with eigenvector
$\phi,$ then
\[
\lambda\phi_{i}=\sum_{j}\frac{\ell_{ij}^{\ast}}{\ell_{ij}}\left(  \phi
_{j}-\phi_{i}\right)  .
\]
We see that
\[
\lambda\sum_{i}\phi_{i}^{2}=-\frac{1}{2}\sum_{i,j}\frac{\ell_{ij}^{\ast}}%
{\ell_{ij}}\left(  \phi_{j}-\phi_{i}\right)  ^{2}%
\]
and so we see immediately that if $\ell_{ij}^{\ast}\geq0,$ then $\lambda
\leq0.$ Furthermore, if the inequality is strict, then $\lambda=0$ implies
that $\phi_{i}=\phi_{j}$ for every edge. On a connected manifold, this implies
that $\phi$ is constant. This type of Laplacian has good numerical properties
and for this reason numerical analysts are often interested in using such a
Laplacian for numerical approximation of PDE. (For instance, see \cite{Kerk}.) 

In two dimensions, the property $\ell_{ij}^{\ast}\geq0$ is a weighted Delaunay
condition \cite{G3}. Note that the argument in the previous paragraph shows
that this condition implies that the Laplacian is negative semidefinite, but
it may have a larger nullspace than just the constant functions. Often in
triangulations of the plane, one gets around the fact that the inequality is
not strict by removing edges with the property that $\ell_{ij}^{\ast}=0$ and
replacing the triangulation with a polygonalization. In the manifold case,
this could potentially introduce curvature to the inside of the polygons, so
we do not pursue this direction. It is not known whether a given piecewise
flat, metrized manifold $\left(  M^{2},T,d\right)  $ can be transformed to
another piecewise flat, metrized manifold $\left(  M^{2},T^{\prime},d^{\prime
}\right)  $ which is weighted Delaunay such that the two induced piecewise
flat manifolds are isometric in a reasonable sense. This is true for the
perpendicular bisector conformal structure, which corresponds to finding
Delaunay triangulations (see \cite{Riv1} and \cite{BS}). In three dimensions,
the property $\ell_{ij}^{\ast}\geq0$ is not equivalent to a weighted Delaunay
condition, and much less is known about the existence of such metrics.
However, the geometric description of the Laplacian ensures that if all the
centers of the highest dimensional simplices are inside those simplices, then
the second property is satisfied (some call this property \textquotedblleft
well-centered,\textquotedblright\ see \cite{DHLM}).

The first property is certainly weaker. There are a number of instances when
one can prove the first property without the second property being true. For
instance, for a metric in a two-dimensional perpendicular bisector conformal
structure, we see that the induced Laplacian is precisely the finite element
Laplacian. This Laplacian always satisfies the first property, but only
satisfies the second if it is Delaunay (see \cite{Riv1} for a proof). We state
a proposition summarizing the known conditions which ensure the first
property. The following proposition is an amalgam of known results.

\begin{proposition}
\label{proposition laplacian definiteness}Let $\left(  M^{n},T,d\right)  $ be
a piecewise flat, metrized manifold. The discrete Laplacian is a negative
semidefinite operator with zero eigenspace corresponding exactly to constant
functions if any of the following are satisfied:

\begin{enumerate}
\item \label{prop label a}$\ell_{ij}^{\ast}>0$ for all edges $\left\{
i,j\right\}  \in E\left(  T\right)  .$

\item \label{prop label c}$n=2$ and the triangulation is in $\mathcal{C}%
^{FI}\left(  M,T,\eta\right)  ,$ a fixed inversive distance conformal
structure, with $\eta_{ij}\geq0$ for all $\left\{  i,j\right\}  \in E\left(
T\right)  .$

\item \label{prop label d}$n=2$ and $d_{ij}>0$ for all $\left(  i,j\right)
\in E_{+}\left(  T\right)  .$

\item \label{prop label e}$n=2$ and $\left(  M,T,d\right)  $ is $\mathcal{C}%
^{PB}\left(  M,T,L\right)  ,$ a perpendicular bisector conformal structure,
for some $L$.

\item \label{prop label f}$n=2$ and for each triangle isometrically embedded
in the plane as $v_{i}v_{j}v_{k},$ the center $c_{ijk}$ is contained within
the circumcircle.

\item \label{prop label g}$n=3$ and $\left(  M,T,d\right)  $ is in
$\mathcal{C}^{P}\left(  M,T\right)  ,$ the sphere packing conformal structure.
\end{enumerate}
\end{proposition}

\begin{proof}
The proofs follow from a number of results from the literature. The fact that
(\ref{prop label a}) implies definiteness is well known in the numerical
analysis community and proven in the discussion before the statement of the
proposition. The fact that (\ref{prop label c}) implies definiteness was
proven for $0\leq\eta_{ij}\leq1$ by Thurston \cite{Thurs} and Marden-Rodin
\cite{MR} and the general case of (\ref{prop label c}) was proven by Guo
\cite{Guo}. In fact, using (\ref{dij for fixed inversive distance}), one
easily sees that (\ref{prop label c}) implies (\ref{prop label d}), and the
fact that (\ref{prop label d}) implies definiteness is in \cite{G3}. We
believe (\ref{prop label d}) implies (\ref{prop label f}), though we have not
verified the proof since there is a direct proof for (\ref{prop label d}).
(\ref{prop label e}) implies definiteness was shown by Rivin \cite{Riv1}.
Also, for (\ref{prop label e}), the center is the circumcenter and thus
(\ref{prop label e}) implies (\ref{prop label f}). The fact that
(\ref{prop label f}) implies definiteness is in \cite{G4}. The fact that
(\ref{prop label f}) implies definiteness follows easily from the definiteness
of the related matrix in the Appendix from \cite{G2} (also in \cite{CR} and
\cite{Riv2}).
\end{proof}

Note that Proposition \ref{proposition laplacian definiteness} only covers a
small subset of the cases one might be interested in. It is of interest that
(\ref{prop label c})-(\ref{prop label g}) are all proven by proving the
definiteness on a single simplex and then extrapolating to the entire complex,
though (\ref{prop label a}) and takes the global structure into account. In
light of (\ref{prop label e}) and (\ref{prop label f}), it may be surprising
that the same are not true, in general, for $n=3$ (one can consider tetrahedra
which are close to flat). It would be of interest to know a condition similar
to (\ref{prop label g}) which implies definiteness for $n=3.$

\subsection{Convexity and rigidity of curvature functionals}

We can use our analysis of the Laplacian to attack two questions about
curvature functionals:

\begin{enumerate}
\item[Q1.] Are the functionals convex?

\item[Q2.] Are critical points rigid?
\end{enumerate}

The first question is more difficult, but if we take first order variations of
$f_{i}$ in two dimensions (i.e., $\frac{d^{2}f_{i}}{dt^{2}}=0$), then we have
the following theorem.

\begin{theorem}
[\cite{Thurs}\cite{MR}\cite{Guo}\cite{Luo1}]\label{theorem 2d convexity}The
function $F$ described in Theorem \ref{2d variational} is convex on the image
of following conformal structures:

\begin{enumerate}
\item $\mathcal{C}^{FI}\left(  M^{2},T,\eta\right)  $, with $\eta_{ij}\geq0$
for all $\left\{  i,j\right\}  \in E\left(  T\right)  .$

\item $\mathcal{C}^{PB}\left(  M^{2},T,L\right)  $, for some $L$.
\end{enumerate}
\end{theorem}

\begin{proof}
The proof follows immediately from Theorem \ref{2d variational} and
Proposition \ref{proposition laplacian definiteness}. This theorem was
previously proven by combining theorems of the articles listed.
\end{proof}

In three dimensions, this question is far more complex, much like in the
smooth case, due to the presence of a reaction term. However, we do have the
following result.

\begin{theorem}
\label{theorem 3d convexity}The Einstein-Hilbert-Regge functional is convex on
the following sets:

\begin{enumerate}
\item Metrics in the image of the conformal structure $\mathcal{C}^{P}\left(
M^{3},T\right)  $ with $K_{i}\geq0$ for all $i\in V\left(  T\right)  .$

\item Metrics in the image of any conformal structure of $\left(
M^{3},T\right)  $ which satisfy $\ell_{ij}^{\ast}-\frac{1}{2}q_{ij}K_{ij}>0$
for each $\left\{  i,j\right\}  \in E\left(  T\right)  $ and $K_{i}\geq0$ for
all $i\in V\left(  T\right)  .$
\end{enumerate}
\end{theorem}

We note that (1) is not a special case of (2). In case (1) we have that
$q_{ij}=0$ but do not require $\ell_{ij}^{\ast}>0$. We also note that a
special case of (2) is a metric in the image of $\mathcal{C}^{FI}\left(
M^{2},T,\eta\right)  $ with $1\leq\eta_{ij}$ and $K_{ij}\geq0$ for all
$\left\{  i,j\right\}  \in E\left(  T\right)  .$

\begin{proof}
Recall the variation formula from Theorem \ref{theorem conformal 3d}. As
already remarked, in case (1) we have $q_{ij}=0.$ Together with case
(\ref{prop label g}) in Proposition \ref{proposition laplacian definiteness},
the case is proven. Case (2) can be proven by essentially the same argument
used to prove Proposition \ref{proposition laplacian definiteness}, part
(\ref{prop label a}).
\end{proof}

The second question above asks about rigidity, which we can define thus.

\begin{definition}
A piecewise flat, metrized manifold $\left(  M,T,d\right)  $ is \emph{rigid}
with respect to conformal variations if there is no conformal variation
$f\left(  t\right)  $ such that $\left(  M,T,d\left(  f\left(  t\right)
\right)  \right)  $ is fixed other than the trivial variation which scales the
edge lengths uniformly (in Riemannian geometry, this is called a homothety).
\end{definition}

Since we have functionals of $\left(  M,T,d\right)  $ in two and three
dimensions, we have the following immediate consequences of Theorems
\ref{2d variational} and \ref{theorem conformal 3d} together with Proposition
\ref{proposition laplacian definiteness}.

\begin{theorem}
\label{theorem 2d rigidity}A two-dimensional piecewise flat, metrized manifold
$\left(  M^{2},T,d\right)  $ with curvature zero (i.e., $K_{i}=0$ for all
$i\in V\left(  T\right)  $) is rigid with respect to any conformal variations
if it satisfies (\ref{prop label a})-(\ref{prop label f}) in Proposition
\ref{proposition laplacian definiteness}.
\end{theorem}

\begin{theorem}
\label{theorem 3d rigidity}A three-dimensional piecewise flat, metrized
manifold $\left(  M^{3},T,d\right)  $ which is Ricci flat is rigid with
respect to any conformal variations if it satisfies (\ref{prop label a}) or
(\ref{prop label g}) in Proposition \ref{proposition laplacian definiteness}.
\end{theorem}

Note that these statements are analogous to a theorem of Obata \cite{Obat} in
the smooth category.

\begin{acknowledgement}
This work benefited from discussions with Mauro Carfora, Dan Champion, and
Feng Luo.
\end{acknowledgement}

\end{document}